\documentclass[oneside,english,11pt]{amsart}
\usepackage{comment}
\usepackage[latin9]{inputenc}
\usepackage{float}
\usepackage{mathtools}
\usepackage{amstext}
\usepackage{amsthm}
\usepackage{amsmath}
\usepackage{amssymb}
\usepackage{geometry}
\usepackage{enumitem,xcolor}
\usepackage[backref=page,colorlinks,citecolor=blue,bookmarks=true]{hyperref}
\usepackage{cleveref}
\hypersetup{linkcolor=blue,  citecolor=blue, urlcolor=blue}
\usepackage[fontsize=11pt]{scrextend}
\usepackage{xspace}

\usepackage[nobysame,abbrev,alphabetic]{amsrefs}
\usepackage{algorithmic}

\usepackage{microtype}
\usepackage{tikz}
\usepackage{circuitikz}
\usepackage{changepage}
\usepackage{nicefrac,xfrac}

\usepackage{mathrsfs}
\floatname{algorithm}{Algorithm}

\makeatletter

\floatstyle{ruled}
\newfloat{algorithm}{tbp}{loa}
\providecommand{\algorithmname}{Algorithm}
\floatname{algorithm}{\protect\algorithmname}

\numberwithin{equation}{section}
\numberwithin{figure}{section}
\theoremstyle{plain}
\newtheorem{thm}{\protect\theoremname}[section]
\DeclareMathOperator*{\Ex}{\mathbb{E}}
\DeclareMathOperator*{\Pro}{\mathbb{P}}

\newtheorem{cor}[thm]{Corollary}
\newtheorem*{thm*}{\protect\theoremname}
\theoremstyle{definition}
\newtheorem{problem}[thm]{\protect\problemname}
\newtheorem*{problem*}{Problem}
\theoremstyle{remark}
\newtheorem*{rem*}{\protect\remarkname}
\theoremstyle{remark}
\newtheorem{rem}[thm]{\protect\remarkname}

\theoremstyle{definition}
\newtheorem{defn}[thm]{\protect\definitionname}
\theoremstyle{plain}

\theoremstyle{plain}
\newtheorem{fact}[thm]{\protect\factname}
\theoremstyle{plain}

\theoremstyle{definition}

\theoremstyle{plain}
\theoremstyle{plain}
\newtheorem{claim}[thm]{Claim}
\theoremstyle{plain}
\usepackage{amsrefs}

\usepackage{algorithmic}

\usepackage{microtype}

\usepackage{changepage}

\floatname{algorithm}{Algorithm}

\let\originalleft\left
\let\originalright\right
\renewcommand{\left}{\mathopen{}\mathclose\bgroup\originalleft}
\renewcommand{\right}{\aftergroup\egroup\originalright}

\usepackage{babel}
\makeatletter
\addto\extrasfrench{%
   \providecommand{\fg}{\ifdim\lastskip>\z@\unskip\fi~\frqq}%
}

\providecommand{\definitionname}{Definition}
\providecommand{\factname}{Fact}
\providecommand{\lemmaname}{Lemma}
\providecommand{\problemname}{Problem}
\providecommand{\propositionname}{Proposition}
\providecommand{\remarkname}{Remark}
\providecommand{\theoremname}{Theorem}
\providecommand{\examplename}{Example}

\global\long\def\eps{\varepsilon}

\newcommand{\erg}{{\rm erg}}

\newcommand{\Sch}{{\rm Sch}}

\newcommand{\IRS}{{\rm IRS}}
\newcommand{\Stab}{{\rm Stab}}
\newcommand{\Id}{{\rm Id}}
\newcommand{\Prob}{{\rm Prob}}

\newcommand{\val}{{\rm val}}
\newcommand{\fd}{{\rm fd}}
\newcommand{\sof}{{\rm sof}}
\newcommand{\surj}{{\rm surj}}
\newcommand{\RBS}{{\rm RBS}}
\newcommand{\dom}{{\rm Dom}}
\newcommand{\cF}{\mathcal{F}}

\newcommand{\RR}{\mathbb{R}}
\newcommand{\QQ}{\mathbb{Q}}
\newcommand{\NN}{\mathbb{N}}

\newcommand{\cP}{\mathcal{P}}
\newcommand{\cT}{\mathcal{T}}
\newcommand{\cQ}{\mathcal{Q}}

\newcommand{\Sub}{\frak{sub}}

\newcommand\Rok{{\rm Rok}}
\newcommand\supp{{\rm supp}}
\newcommand\perc{{\rm perc}}

\newcommand{\scrF}{\mathscr{F}}
\newcommand{\scrI}{\mathscr{I}}
\newcommand{\scrB}{\mathscr{B}}

\def\cc{{\curvearrowright}}
\def\G{{\Gamma}}

\newcommand{\mnote}[1]{\textcolor{red}{\footnotesize{\bf (Michael:} {#1}{\bf ) }}}

\title{Surjunctivity does not characterize cosoficity of invariant random subgroups}

\author[L.\ Bowen]{Lewis Bowen}
\address{Lewis Bowen\hfill\break
Department of Mathematics \hfill\break 1 University Station C1200 \hfill\break University of Texas at Austin\hfill\break Austin, TX, 78712 USA.}
\email{lpbowen@math.utexas.edu}

\author[M.\ Chapman]{Michael Chapman}
\address{Michael Chapman\hfill\break
	School of Mathematics\hfill\break
	Institute for Advanced Study,\hfill\break 1 Einstein Drive, Princeton, NJ 08540, USA.}
\email{mchapman@ias.edu}

\begin{document}
\begin{abstract}
A  group is \emph{surjunctive} if  every injective cellular automaton on it is also surjective.  Gottschalk famously conjectured that  all groups are surjunctive. This remains a central open problem in symbolic dynamics and descriptive set theory. 
     Gromov and Weiss  termed the notion of \emph{sofic} groups, and proved that all such groups are {surjunctive}, providing the largest class of groups which satisfy Gottschalk's conjecture. It is still open to decide whether all groups are sofic. This became a major open problem in group theory, and is related to other well known problems such as the Aldous--Lyons conjecture in probability theory and to Connes' embedding problem in the theory of operator algebras.
     
     A complementary natural question to ask is:
     Does the reverse implication to Gromov and Weiss' result holds? Namely, are all surjunctive groups sofic?  As currently there are no known non-sofic groups, answering this problem in the negative in the category of groups is still out of reach. This paper  resolves this problem in the   generalized setup of \emph{invariant random subgroups} of free groups (IRSs), where non (co)sofic objects were recently shown to exist by Lubotzky, Vidick and the two authors. Specifically, we prove that there \emph{exists} a surjunctive non (co)sofic IRS, resolving the aforementioned problem in the negative. Our proof uses a complexity theoretic approach, and in particular a recent development due to Manzoor, as well as the theory of Rokhlin entropy developed by Seward and others. As a byproduct of our proof technique, the non (co)sofic IRS we provide satisfies a condition stronger  than  surjunctivity; it satisfies a version of Seward's \emph{maximal Rokhlin entropy of Bernoulli Shifts} (RBS) criterion.
\end{abstract}
\maketitle

\section{\textbf{Introduction}}

Let $\Gamma$ be a finitely generated group,\footnote{Often the more generalized setup of discrete countable groups is used. But for the simplicity of various definitions, we prefer to keep the finite generation assumption throughout.}  $\Sigma$ a finite set (of colors), and $\Sigma^\Gamma$ the $\Sigma$-colorings of $\Gamma$. A \emph{cellular automaton}  on $\Gamma$ with color palette $\Sigma$ is a local, element independent recoloring mechanism of $\Gamma$; namely, it is a recoloring map $\Phi\colon \Sigma^\Gamma\to \Sigma^\Gamma$, and there are elements $\gamma_1,...,\gamma_n$ and a function $\phi\colon \Sigma^n\to \Sigma$ such that $\Phi(c)(x)=\phi(c(x\gamma_1),...,c(x\gamma_n))$ for every coloring $c\colon \Gamma\to \Sigma$ and element $x\in \Gamma$. 
A group $\Gamma$ is said to be \emph{surjunctive} if every injective cellular automaton on it  is also surjective.    
A  well known open problem due to Gottschalk \cite{gottschalk2006some} is:
    Are all finitely generated groups surjunctive?
In \cite{MR1694588} and \cite{weiss-2000}, Gromov and Weiss defined the notion of a \emph{sofic} group (see Definition \ref{defn:soficity_and_cosoficity}), and showed that all such groups are surjunctive. Subsequently, the following became a major open problem:
    Are all finitely generated groups sofic?
Another natural problem that arises from their work is:
\begin{problem}[\cites{MR1694588,weiss-2000}]\label{problem:characterize}
    Are all surjunctive groups sofic? Namely, does surjunctivity characterize soficity?
\end{problem}
Currently, no non-sofic group is known to exist, so a negative answer to Problem \ref{problem:characterize} is fairly out of reach.
But, there is a generalized setup to this one, in which both soficity and surjunctivity are well defined, and also non sofic objects were recently shown to exist \cite{BCLV_subgroup_tests,Tailored_MIPRE} --- this is the setup of \emph{invariant random subgroups} (IRSs) of free groups.\footnote{The study of invariant random subgroups was initiated independently by \cite{MR2749291}, \cite{abert2014kesten} and \cite{MR3193754} (and appeared implicitly in \cite{stuck1994stabilizers}). It has become a fruitful area of research, as several open problems in group theory, number theory and dynamics were resolved by studying IRSs (cf.\ \cite{Gelander_ICM2018, MR3664810,fraczyk2023infinite} and the references therein).}   

An invariant random subgroup of a given group is a probability measure over its subgroups which is conjugate invariant (see Section \ref{sec:IRSs} for a detailed definition). Every pair $(\Gamma,S)$ consisting of a group and a set of generators for it induces a short exact sequence $N\hookrightarrow \cF\twoheadrightarrow \Gamma$, where $\cF=\cF_S$ is  the free group with basis $S$. As ${\bf 1}_N$, the Dirac measure concentrated at $N$, is an IRS of $\cF$, the mapping $(\Gamma,S)\mapsto {\bf 1}_N$ embeds the finitely generated groups  into IRSs of free groups. 
This embedding is meaningful in our context, as IRSs can be \emph{cosofic} (see Definition \ref{defn:soficity_and_cosoficity}) and surjunctive (see 
Definition \ref{defn:surjunctive}), and ${\bf 1}_N$ is cosofic (respectively, surjunctive\footnote{It would be more natural to call this property \emph{cosurjunctivity}, but we decided to avoid this.}) if and only if $\Gamma\cong \nicefrac{\cF}{N}$ is sofic (respectively, surjunctive). Moreover, similar to the result of Gromov and Weiss, one can show (see Corollary \ref{cor:cosofic_IRSs_satisfy_RBS} and Theorem \ref{thm:RBS_implies_surj}) that a cosofic IRS is always surjunctive. 
So, the analogoue of Problem \ref{problem:characterize} arises --- are all surjunctive IRSs of free groups cosofic? Our main result is a negative answer to this problem:

\begin{thm}\label{thm:main}
    There exists a surjunctive non cosofic IRS of a finitely generated free group.
\end{thm}

In the rest of this introduction, we provide more detailed definitions of cosoficity and surjunctivity, as well as describe our proof technique. 

\subsection{Cosoficity and Surjunctivity  of Invariant Random Subgroups}\label{sec:IRSs}
 
 The set  $\Sub(\cF)$ consisting of all subgroups of the free group $\cF$ is compact when equipped with the product topology. A \emph{random subgroup} is a Borel probability distribution over $\Sub(\cF)$, and we denote the collection of all of them by $\Prob(\Sub(\cF))$.\footnote{It is often the case that a random subgroup is a random variable with values in subgroups of $\cF$ that  is distributed according to some law. We do not use this perspective at all along the paper, and a random subgroup (and later random subset) is the distribution law itself.} The random subgroups can be equipped with the weak* topology,
and with this topology,  
this space is compact.
The conjugation action $w.K=wKw^{-1}$ of $\cF$ on its subgroups extends naturally to random subgroups.\footnote{Because $H$ will be used throughout this paper to denote entropy (of various kinds), we avoid using this notation for subgroups and use $K$ for subgroups instead.}
\begin{defn}[Invariant Random Subgroups]\label{defn:IRS}
    An \emph{invariant random subgroup} (IRS) of $\cF$ is a random subgroup $\pi\in \Prob(\Sub(\cF))$ such that $w.\pi=\pi$ for every $w\in \cF$. The set of all IRSs of $\cF$ is denoted by $\IRS(\cF)$.
\end{defn}

For us there will be two main examples of IRSs. First, for every normal subgroup $N\trianglelefteq \cF$, the Dirac measure ${\bf 1}_N$ is an IRS. In addition, whenever $\cF$ acts on a finite set $\Omega$, it induces a \emph{finitely described} IRS by choosing a uniformly random stabilizer of this action; namely, letting $\Stab(x)=\{w\in \cF\mid w.x=x\}$, the IRS associated with $\cF\cc \Omega$  is 
\begin{equation}\label{eq:fin_desc_IRS}
\pi= \frac{1}{|\Omega|}\sum_{x\in \Omega}{\bf 1}_{\Stab(x)}.    
\end{equation}
\begin{defn}[Soficity and Cosoficity]\label{defn:soficity_and_cosoficity}
    An IRS $\pi$ is \emph{cosofic} if it is in the weak* closure of the finitely described IRSs, and we denote the set of these by $\IRS_\sof(\cF)$.
    A finitely generated group $\Gamma$ is \emph{sofic} if it is isomorphic to $\nicefrac{\cF}{N}$ with $\cF$ a finitely generated free group, $N\trianglelefteq \cF$, and ${\bf 1}_N$ a cosofic IRS of $\cF$.
\end{defn}

The $\Sigma$-colorings of a group $\Gamma$ can be equipped with the product topology. In addition, $\Gamma$ 
acts naturally  on $\Sigma^\Gamma$ by shifts; namely for $\gamma\in \Gamma$ and $c\colon \Gamma\to \Sigma$,  $\gamma.c(x)=c(\gamma^{-1}x)$. With these two observations, a cellular automaton is a recoloring map $\Phi\colon \Sigma^\Gamma\to \Sigma^\Gamma$ which is continuous and commutes with the group action, i.e., $\gamma.\Phi (c)=\Phi(\gamma.c)$ for every $\gamma\in \Gamma$ --- this property is often called \emph{conjugacy invariance} or $\Gamma$-\emph{equivariance}. We now generalize this to IRSs. 
Given a subgroup  $K \le \cF$, a coloring $c \colon \cF\to \Sigma$ is $K$-invariant if $k.c=c$ for all $k\in K$ --- this data can also be given as a $\Sigma$-coloring of the right $K$-cosets $_K\setminus^\cF=\{Kx\mid x\in \cF\}$. 
Let   
$\Sub(\cF,\Sigma)$ be the set of all $\Sigma$-colorings of  $K$-cosets of $\cF$ for any subgroup $K\leq \cF$;  namely, it consists of all pairs $(K,c)$ where $K$ is a subgroup of $\cF$ and $c\colon \cF\to \Sigma$ is $K$-invariant.
As $\Sub(\cF,\Sigma)\subseteq \Sub(\cF)\times \Sigma^\cF$, it can be equipped with the product topology. 
In addition, the diagonal action $w.(K,c) = (wKw^{-1}, w.c)$ of $\cF$ on $\Sub(\cF)\times \Sigma^\cF$ preserves $\Sub(\cF,\Sigma)$ ---  as, $c$ being $K$-invariant implies $w.c$ is $wKw^{-1}$-invariant.
  For $\pi$ an IRS of $\cF$, let ${\rm supp}(\pi)\subseteq \Sub(\cF)$ be its support --- i.e., subgroups whose neighborhoods all have positive $\pi$-measure. Define the \emph{domain of $\pi$}, $\dom(\pi)\subseteq \Sub(\cF,\Sigma)$, to consist of all pairs $(K,c)$ where $K\in {\rm supp}(\pi)$    --- namely, all colorings of  right $K$-cosets of subgroups that $\pi$ can ``actually sample''. Since $\pi$ is an IRS,  $\dom(\pi)$ is preserved under the action of $\cF$. Finally, let $P\colon \Sub(\cF,\Sigma)\to \Sub(\cF)$ be the projection to the first coordinate. 
\begin{defn}[Cellular automaton on an IRS]\label{defn:cellular_automaton_on_IRS}
  A function $\Phi\colon \dom(\pi)  \to \dom(\pi)$ is a cellular automaton on $\pi$ (with color palette $\Sigma$) if it is continuous, conjugate invariant (i.e., $w\Phi w^{-1}(K,c)=\Phi(K,c)$ for every $(K,c)\in \dom(\pi)$ and $w\in \cF$) and $P\circ \Phi=P$ (i.e., the $\Phi$-image of $(K,c)$ must be $(K,c')$, so $\Phi$ indeed recolors the right cosets of $K$  for every $K$ in the support of $\pi$).
\end{defn}

\begin{defn}[Surjunctive IRS]\label{defn:surjunctive}
    A $\pi\in \IRS(\cF)$ is \emph{surjunctive} if every (almost everywhere) injective $\pi$-cellular automaton must be surjective.
\end{defn}

\begin{rem}
    A  cellular automaton on an IRS $\pi$ is a generalization of those defined over finitely generated groups (when taking $\pi={\bf 1}_N$ with $\Gamma\cong \cF/N$). In addition,  ${\bf 1}_N$ is surjunctive as in Definition \ref{defn:surjunctive} if and only if $\nicefrac{\cF}{N}$ is a surjunctive group. 
\end{rem}
\begin{center}
    The reader can now parse Theorem \ref{thm:main}.
\end{center}
\ 
\\

\subsection{Proof technique}
Describing our proof requires three steps:
\begin{enumerate}
    \item First, we need to recall the argument structure of \cites{BCLV_subgroup_tests,Tailored_MIPRE}, where non cosofic IRSs were shown to exist.
    \item Then, we need to describe Manzoor's recent work \cite{manzoor2025invariant}, in which he used a certain amplification of the above argument structure to deduce the existence of non cosofic IRSs which satisfy L\"uck's determinant conjecture \cite{luck2002l2}.
    \item Finally, to implement the appropriate amplification a la Manzoor, we use an IRS version of a sufficient condition to surjunctivity due to Seward \cite{seward2019krieger}.
\end{enumerate}

\begin{figure}[!httb]
\centering
\begin{circuitikz}[scale=0.4]
\tikzstyle{every node}=[font=\tiny]
\draw  (16.25,10.75) circle (6.25cm);

\draw  (16.25,10.75) circle (3.75cm);
\draw [->, >=Stealth, dashed] (16.25,19.5) -- (16.25,17.25);
\draw [->, >=Stealth, dashed] (25,10.75) -- (22.5,10.75);
\draw [->, >=Stealth, dashed] (20,18.25) -- (18.75,16.5);
\draw [->, >=Stealth, dashed] (23.75,14.5) -- (22,13.25);
\draw [->, >=Stealth, dashed] (22.5,17) -- (20.75,15.25);
\draw [->, >=Stealth, dashed] (12.5,18.25) -- (13.75,16.5);
\draw [->, >=Stealth, dashed] (7.5,10.75) -- (10,10.75);
\draw [->, >=Stealth, dashed] (8.75,7) -- (10.5,8.25);
\draw [->, >=Stealth, dashed] (16.25,2) -- (16.25,4.25);
\draw [->, >=Stealth, dashed] (12.5,3.25) -- (13.75,5);
\draw [->, >=Stealth, dashed] (20,3.25) -- (18.75,5);
\draw [->, >=Stealth, dashed] (10,4.5) -- (11.75,6.25);
\draw [->, >=Stealth, dashed] (22.5,4.5) -- (20.75,6.25);
\draw [->, >=Stealth, dashed] (23.75,7) -- (22,8.25);
\draw [->, >=Stealth, dashed] (8.75,14.5) -- (10.5,13.25);
\draw [->, >=Stealth, dashed] (10,17) -- (11.75,15.25);
\node [ color={rgb,255:red,255; green,0; blue,0}] at (16,20) {Outer approximation};
\node  at (16.25,16) {All IRSs};
\node at (16.25,10.75) {Cosofic IRSs};
\draw [->, >=Stealth, dashed] (16.25,13.25) -- (16.25,14.5);
\draw [->, >=Stealth, dashed] (13.75,10.75) -- (12.5,10.75);
\draw [->, >=Stealth, dashed] (16.25,8.25) -- (16.25,7);
\draw [->, >=Stealth, dashed] (18.75,10.75) -- (20,10.75);
\draw [->, >=Stealth, dashed] (17.5,12) -- (19,13.25);
\draw [->, >=Stealth, dashed] (15,12) -- (13.5,13.25);
\draw [->, >=Stealth, dashed] (17.5,9.5) -- (19,8.25);
\draw [->, >=Stealth, dashed] (15,9.5) -- (13.5,8.25);
\node [color={rgb,255:red,255; green,0; blue,0}] at (16.25,13) {Inner};
\node [color={rgb,255:red,255; green,0; blue,0}] at (16.25,12) {approximation};
\end{circuitikz}

\caption{Outer and inner algorithmic approximations, together with an incomputability result, imply that the two sets are different.}\label{fig:IRSs}
\end{figure}
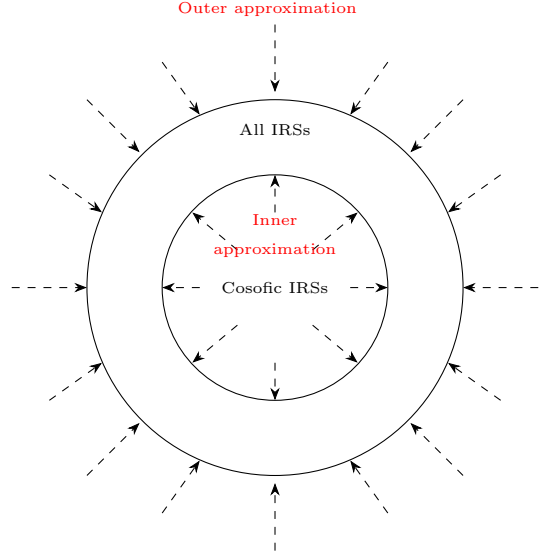

\subsubsection{Subgroup Tests and non cosofic IRSs}\label{sec:subgroup_tests}
A \emph{challenge} is a continuous function $D$ from $\Sub(\cF)$ to $\{0,1\}$. A \emph{subgroup test} $\cT$ is a rational probability distribution over finitely many challenges, namely it consists of challenges $D_1,...,D_m$ and $\mu\in \Prob([m])\cap \QQ^m$. The \emph{value} of a random subgroup $\pi$ versus a subgroup test $\cT$   is 
\begin{equation}\label{eq:value_of_a_subgroup_test}
\val(\cT,\pi)=\Ex_{K\sim \pi}\Ex_{i\sim \mu}[D_i(K)].\end{equation}
It is immediate that $\val(\cT,\cdot)$ is a continuous linear functional on the random subgroups. The sofic value $\val_\sof(\cT)$ of a subgroup test $\cT$  is the supremum of $\val(\cT,\pi)$ over the finitely described IRSs, and its ergodic value $\val_\erg(\cT)$ is the supremum of $\val(\cT,\pi)$ over all  IRSs. By the continuity of $\val(\cT,\cdot)$, if  every IRS is cosofic, then $\val_\sof(\cT)=\val_\erg(\cT)$. 
Given $\cT$ and a finitely described IRS $\pi$, one can  calculate $\val(\cT,\pi)$. As there are countably many finitely described IRSs, one can list them $\{\pi_i\}_{i=1}^\infty$ and evaluate $\val(\cT,\pi_i)$ sequentially. The limsup of this sequence is $\val_\sof(\cT)$, and thus:
\begin{fact}\label{fact:inner}
There is an algorithm that approximates $\val_\sof(\cT)$, the sofic value of a subgroup test $\cT$, from below.     
\end{fact}
On the other hand, in \cite{BCLV_subgroup_tests} we provided a downwards sequence of  polytopes (with finitely many faces) $\{\cQ_r\}_{r=1}^\infty$ in $\Prob(\Sub(\cF))$ whose intersection is $\IRS(\cF)$.\footnote{These polytopes are the pseudo IRSs with respect to finite subsets. The definition is recalled in Section \ref{sec:outer_approximation}.} Using linear programming, one can evaluate $\beta_r=\displaystyle{\sup_{\pi \in \cQ_r} }\val(\cT,\pi)$  for each $r$ in finite time. Since $\displaystyle{\liminf_{r\to \infty}}\beta_r=\val_\erg(\cT)$, we have:
\begin{fact}\label{fact:outer}
    There is an algorithm that approximates $\val_\erg(\cT)$, the ergodic value of a subgroup test $\cT$, from above.\footnote{Algorithms like this one are commonly called ``NPA type hierarchies'', due to \cite{navascues2008convergent}.}    
\end{fact}
Combining Facts \ref{fact:inner} and \ref{fact:outer}, if $\val_\sof(\cT)=\val_\erg(\cT)$ for every $\cT$, then $\val_\sof (\cT)$ can be calculated up to any predetermined accuracy.  A theorem in \cite{BCLV_subgroup_tests} states that:
\begin{fact}\label{fact:sof_val_is_RE_hard}
    The following task is as hard as the Halting Problem: Given a subgroup test $\cT$ and a positive rational $\eps$, output a rational $v$ such that $|\val_\sof(\cT)-v|\leq \eps$.
\end{fact}

Combining all of the above, one deduces that there must be a subgroup test $\cT$ satisfying $\val_\sof(\cT)< \val_\erg(\cT)$, which in turn implies non cosofic IRSs exist. See Figure \ref{fig:IRSs} for a visualization of this setup.

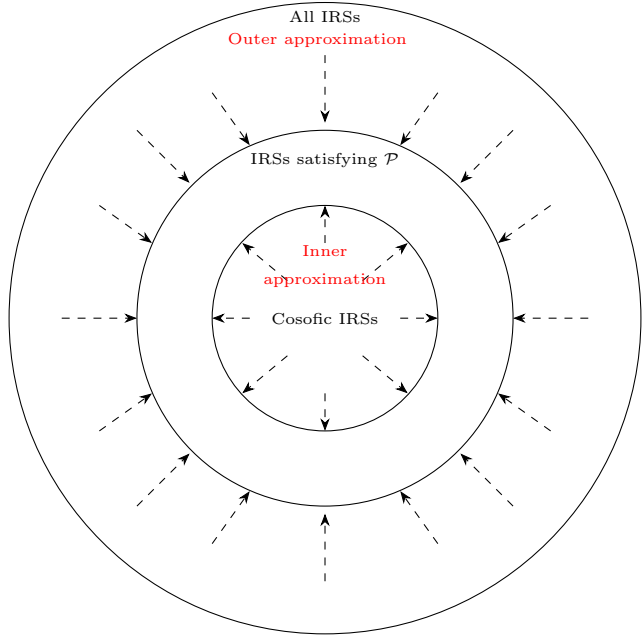
\begin{figure}[!httb]
\centering
\begin{circuitikz}[scale=0.4]
\tikzstyle{every node}=[font=\tiny]
\draw  (16.25,10.75) circle (10.5cm);
\draw  (16.25,10.75) circle (6.25cm);

\draw  (16.25,10.75) circle (3.75cm);
\draw [->, >=Stealth, dashed] (16.25,19.5) -- (16.25,17.25);
\draw [->, >=Stealth, dashed] (25,10.75) -- (22.5,10.75);
\draw [->, >=Stealth, dashed] (20,18.25) -- (18.75,16.5);
\draw [->, >=Stealth, dashed] (23.75,14.5) -- (22,13.25);
\draw [->, >=Stealth, dashed] (22.5,17) -- (20.75,15.25);
\draw [->, >=Stealth, dashed] (12.5,18.25) -- (13.75,16.5);
\draw [->, >=Stealth, dashed] (7.5,10.75) -- (10,10.75);
\draw [->, >=Stealth, dashed] (8.75,7) -- (10.5,8.25);
\draw [->, >=Stealth, dashed] (16.25,2) -- (16.25,4.25);
\draw [->, >=Stealth, dashed] (12.5,3.25) -- (13.75,5);
\draw [->, >=Stealth, dashed] (20,3.25) -- (18.75,5);
\draw [->, >=Stealth, dashed] (10,4.5) -- (11.75,6.25);
\draw [->, >=Stealth, dashed] (22.5,4.5) -- (20.75,6.25);
\draw [->, >=Stealth, dashed] (23.75,7) -- (22,8.25);
\draw [->, >=Stealth, dashed] (8.75,14.5) -- (10.5,13.25);
\draw [->, >=Stealth, dashed] (10,17) -- (11.75,15.25);
\node [ color={rgb,255:red,255; green,0; blue,0}] at (16,20) {Outer approximation};
\node  at (16.25,16) {IRSs satisfying $\cP$};
\node  at (16.25,20.8) {All IRSs};
\node at (16.25,10.75) {Cosofic IRSs};
\draw [->, >=Stealth, dashed] (16.25,13.25) -- (16.25,14.5);
\draw [->, >=Stealth, dashed] (13.75,10.75) -- (12.5,10.75);
\draw [->, >=Stealth, dashed] (16.25,8.25) -- (16.25,7);
\draw [->, >=Stealth, dashed] (18.75,10.75) -- (20,10.75);
\draw [->, >=Stealth, dashed] (17.5,12) -- (19,13.25);
\draw [->, >=Stealth, dashed] (15,12) -- (13.5,13.25);
\draw [->, >=Stealth, dashed] (17.5,9.5) -- (19,8.25);
\draw [->, >=Stealth, dashed] (15,9.5) -- (13.5,8.25);
\node [color={rgb,255:red,255; green,0; blue,0}] at (16.25,13) {Inner};
\node [color={rgb,255:red,255; green,0; blue,0}] at (16.25,12) {approximation};
\end{circuitikz}

\caption{Outer  algorithmic approximation for any set of IRSs which satisfy a certain property $\cP$, together with the incomputability result for $\val_\sof(\cT)$ from Fact \ref{fact:sof_val_is_RE_hard}, imply that  there are non cosofic IRSs with this property $\cP$.}\label{fig:IRSs2}
\end{figure}

\subsubsection{Manzoor's amplification of the complexity theoretic framework}
Let $\cP$ be some property satisfied by all cosofic IRSs. Let $\IRS_\cP(\cF)$ be the set of all IRSs that satisfy this property. Then, this set sits in between the cosofic IRSs and all IRSs --- see Figure \ref{fig:IRSs2}. If this set has a natural outer approximation, one deduces immediately from Facts \ref{fact:sof_val_is_RE_hard} and \ref{fact:inner} that there \textbf{exist} non cosofic IRSs \textbf{which satisfy} this property. This is how Manzoor \cite{manzoor2025invariant} found a non cosofic IRS which satisfies L\"uck's determinant conjecture \cite{luck2002l2}. Hence, it is natural to choose $\IRS_\surj(\cF)$ to be the set of all surjunctive IRSs, and find an outer approximation of this set using polytopes. Unfortunately, we do not know how to do that. But, we can find a \textbf{subset} of $\IRS_\surj(\cF)$  that has an outer approximation, which will be sufficient to deduce Theorem \ref{thm:main}. 
\begin{rem}
    Manzoor suggests in  \cite{manzoor2025invariant} to study both outer  and inner approximations for $\IRS_\cP(\cF)$. We believe that approximating $\val_\erg(\cT)$ should be impossible (similar to $\val_\sof(\cT)$ in Fact \ref{fact:sof_val_is_RE_hard}; the non-local games analogue of this was recently proved in \cite{lin2025mipco}), and if true, inner approximations should yield IRSs that do not satisfy $\cP$. Specifically, we neither know whether $\IRS_\surj(\cF)$ has   an outer approximation, nor an inner approximation, but if for some reason it has an inner approximation, together with an inapproximability result for $\val_\erg(\cT)$ one will deduce the existence of a non-surjunctive IRS, which would be very exciting. 
\end{rem}
\begin{rem}
    Our Theorem \ref{thm:main} can probably be upgraded to the exsitence of a surjunctive non cohyperlinear  IRS (which is stronger than being non cosofic)   using the methods of \cite{manzoor2025there,manzoor2025invariant}. In that case, the combination of  the above strengthening of Theorem \ref{thm:main} together with Theorem 1.3 in \cite{manzoor2025invariant} would provide a non cohyperlinear IRS that satisfies both Gottschalk's conjecture and L\"ucks determinant conjecture simultaneously. This shows another beautiful aspect of Manzoor's approach, as it has a cumulative nature.  
\end{rem}

\subsubsection{Seward's Rokhlin entropy criterion for surjunctivity}

In \cite{seward2019krieger}, Seward defined what it means for a group $\Gamma$ to satisfy the \emph{maximal Rokhlin entropy on Bernoulli Shifts criterion} (RBS), and proved that a $\Gamma$ satisfying RBS must be surjunctive. Here, we suggest an IRS analogue of RBS, and prove that indeed every IRS which satisfies RBS is also surjunctive (Theorem \ref{thm:RBS_implies_surj}). Similarly, every cosofic IRS satisfies RBS (Corollary \ref{cor:cosofic_IRSs_satisfy_RBS}). In summary,  if we let $\IRS_{\RBS}(\cF)$ be the collection of all IRSs satisfying RBS, then we have
\begin{equation}\label{eq:soficIRS_in_RBSIRS_in_surjunctiveIRS}
\IRS_\sof(\cF)\subseteq \IRS_\RBS(\cF)\subseteq \IRS_\surj (\cF).
\end{equation}
To implement Manzoor's approach, we need to find an outer approximation to $\IRS_\RBS(\cF)$. Given such an outer approximation, by the previous discussion,  a non cosofic IRS which satisfies RBS exists, and in turn it will be surjunctive. In Section \ref{sec:outer_approximation} we provide the required outer approximation. 

Let us demystify the RBS criterion a bit. To that end, we  describe first the modern approach to Gromov and Weiss' result that ``sofic groups are surjunctive'' (cf.\ Section 5 of \cite{bowen2018brief}).
Following the first author's work in the measurable setup \cite{bowen2010measure}, Kerr and Li \cite{MR2854085} defined the notion of \emph{topological sofic entropy} of  an action $\Gamma\cc X$  of a sofic group on a compact topological space.\footnote{Actually, the sofic entropy is defined with respect to a specific sofic approximation of the group $\Gamma$. But, as this explanation aims to provide only ``the ideas'', we ignore this intricacy.} This quantity satisfies the following three facts:
$(1)$ As discussed before, for every finite set $\Sigma$ one has the shift action of $\Gamma$ on $\Sigma^\Gamma$ (which is a compact space), and the sofic entropy of this action is $\log|\Sigma|$. 
$(2)$ In addition, for every closed  strict subset $X\subsetneq \Sigma^\Gamma$ which is $\Gamma$-invariant, the action $\Gamma\cc X$ has sofic entropy which is strictly smaller than $\log |\Sigma|$.
    $(3)$ Moreover, sofic entropy is a \emph{topological invariant}: Given two actions $\Gamma\cc X$ and $\Gamma \cc Y$, a (topological) \emph{factor map} between them is a continuous $\Phi\colon X\to Y$ which commutes with the group action. In case  a factor map $\Phi$ is injective, the sofic entropy of $\Gamma\cc X$ is smaller or equal than that of $\Gamma\cc Y$.  \quad    
Hence, given a non surjective cellular automaton $\Phi\colon \Sigma^\Gamma\to \Sigma^\Gamma$ (which is a factor map by definition), the action of $\Gamma$  on ${\rm Im}(\Phi)\subseteq \Sigma^\Gamma$ has sofic entropy strictly smaller than $\log|\Sigma|$ by clause $(2)$.  On the other hand, if $\Phi$ is injective, then by clauses $(1)$ and $(3)$ the sofic entropy of this action is at least $\log|\Sigma|$. Hence, an injective cellular automaton on a sofic group must be surjective.  

Seward developed in \cite{seward2019krieger1} a way of associating a new type of entropy, which he termed \emph{Rokhlin entropy}, to a probability measure preserving (p.m.p) action of any finitely generated group $\Gamma$ (not necessarily a sofic one) on a standard probability space $(X,\mu)$, which generalizes the Kolmogorov--Sinai entropy in the amenable case (similar to sofic entropy). By equipping $\Sigma$ with a probability measure $\kappa$, and $\Sigma^\Gamma$ with the product measure $\kappa^\Gamma$, the shift action of $\Gamma$ on $\Sigma^\Gamma$ becomes p.m.p --- this action is known as the  \emph{Bernoulli shift} of $\Gamma$ with base $(\Sigma,\kappa)$.
The measurable setup has the appropriate analogue of a factor map (with continuity replaced by measurability), and the analogue of the above clause $(3)$ is satisfied for it. 
The {RBS criterion} is an analogue of $(1)$, namely, a group $\Gamma$ satisfies this criterion if the Rokhlin entropy of the Bernoulli shift with base $(\Sigma,\kappa)$ is equal to the Shannon entropy of the base $H(\kappa)=-\displaystyle{\sum_{a\in \Sigma}}\kappa(a)\log(\kappa(a))$, which is its maximal possible value --- note that in case $\kappa$ is the uniform distribution $u$, this quantity is exactly $\log|\Sigma|$, as it was in the sofic case. 
Under the assumption of RBS, one can  also recover a  slightly more intricate variant of $(2)$.
Finally, one needs to note that a cellular automaton induces a measurable factor map from $(\Sigma^\Gamma,u^\Gamma)$, and by applying the appropriate version of the above argument (see Corollary 4.1 in \cite{seward2019krieger}),   deduce that a group satisfying RBS must be surjunctive.

\subsection{Structure of the paper}
Section \ref{sec:entropy_theory} is devoted to the basics of entropy theory as well as the more intricate theory of Rokhlin entropy. Section \ref{sec:RBS} defines the maximal Rokhlin entropy of Bernoulli Shifts (RBS) criterion, which was originally defined for groups, in the generalized setup of invariant random subgroups (IRSs). Section \ref{sec:cont_cond} provides an equivalent criterion to RBS which is more amenable to analysis. Section \ref{sec:outer_approximation} provides the desired outer approximation result for IRSs which satisfy RBS. Finally, in Section \ref{sec:proof_of_main} we collect everything and prove Theorem \ref{thm:main}.

\subsection*{Acknowledgments}
We would like to thank Brandon Seward for helpful discussions, and for suggesting to us to use Fact \ref{fact:cor7.3Seward} to prove Theorem \ref{thm:RBS_implies_surj}. We would also like to thank Aareyan Manzoor  for providing us with useful remarks that were incorporated to the final presentation, and in particular the need for a rational approximation as  devised in Fact \ref{fact:rational_outer_approximation}.  
LB is partially supported by NSF Grant DMS--2453399. MC would like to thank the Institute for Advanced Study for the fruitful environment, and is  supported by  NSF 
Grant DMS--2424441.

\section{\textbf{Entropy Theory}}\label{sec:entropy_theory}

 We need some basic facts from classic probability theory, ergodic theory and entropy theory. An expert could skip most of Section \ref{sec:entropy_theory}, referring back to it for notation if necessary. 
We use  Section 2 of \cite{seward2016weak}  as a guide, and the reader can refer there, as well as to the Appendix of \cite{seward2019krieger},
for any missing information.

\subsection{Basic Notions}\label{sec:basic_notions}

    Let $(X,\scrB_X,\mu)$ be a standard probability space. Throughout the paper, we assume $X$ is already equipped with a metric which makes it complete and separable (namely, a Polish space) so that $\scrB_X$ is the Borel sigma-algebra of $X$ with respect to this metric --- the notation $\scrB_X$ for the Borel sigma-algebra of a topological space will be used throughout. We denote by $x\sim \mu$ a variable $x$ sampled according to $\mu$. 
    Given a Borel measurable map  between two standard Borel spaces $\psi\colon (X,\scrB_X)\to (Y,\scrB_Y)$, $Y$ becomes a standard probability space by equipping it with the pushforward measure $\nu=\psi_* \mu$, defined by $\psi_*\mu (B)=\mu(\psi^{-1}(B))$ for every Borel set $B\in \scrB_Y$. By the disintegration of measures theorem (cf.\ Theorem 5.14 in \cite{Einsiedler2011}), for $\nu$-almost every $y\in Y$, there is a Borel probability measure $\mu^\psi_y$ on  $X$ which is supported on $\psi^{-1}(y)$, and for every Borel measurable map $f\colon X\to \mathbb{R}$ we have $\displaystyle{\Ex_{x\sim \mu}[f(x)]=\Ex_{y\sim \nu}\Ex_{x\sim \mu^\psi_y}[f(x)]}$. The measures $\mu^\psi_y$ are the \emph{conditional distributions} $\mu^\psi_y(\cdot)=\displaystyle{\Pro_{x\sim \mu}[x\in \cdot |\psi(x)=y]}$.
    Note that a measurable map $\psi\colon (X,\scrB_X)\to (Y,\scrB_Y)$ as before induces on $X$ a sub sigma-algebra $\psi^{-1}(\scrB_Y)$. Conversely, every sub sigma-algebra $\scrF\subseteq \scrB_X$  induces a factor map $\psi=\psi_\scrF\colon (X,\scrB_X)\to (Y,\scrB_Y)$ which satisfies $\psi^{-1}(\scrB_Y)=\scrF$ up to $\mu$-null sets. Hence, conditioning over a sub sigma-algebra makes sense. 
    Let $(\Sigma,\kappa)$ be a {finite} or countable probability space. The \emph{Shannon entropy} of $\kappa$ is $H(\kappa)=-\sum_{a\in \Sigma}\kappa(a)\log(\kappa(a))\in [0,\infty]$ --- this is the  expected amount of information (in bits) required to describe the outcome of a sample according to $\kappa$.    

    \subsection{Observables and Partitions}\label{sec:observables}
    Let $(X,\scrB_X,\mu)$ be a standard probability space, and $\Sigma$  a {finite} or countable set equipped with the discrete topology ---  $\Sigma$ will often be called the \emph{color pallette, alphabet} or \emph{indexing set}. An \emph{observable} is a measurable map $\phi:X \to \Sigma$; the collection of inverse images $\{\phi^{-1}(a)\mid a\in \Sigma\}$ is often called a (measurable) \emph{partition} of $X$ indexed by $\Sigma$. 
    By pushforwarding $\mu$ along $\phi$, one gets a probability distribution $\phi_*\mu$ over $\Sigma$; we denote the Shannon entropy of this measure by $H_\mu(\phi)=H(\phi_*\mu)$, and  call this the Shannon entropy of $\phi$ with respect to $\mu$ --- this is the expected amount of information gained from learning  $\phi(x)$ for $x\sim \mu$. Given a sub sigma algebra $\scrF\subseteq \scrB_X$, and letting $\psi\colon (X,\scrB_X)\to (Y,\scrB_Y)$ be the associated factor and $\nu=\psi_*\mu$ the pushforward measure, the \emph{conditional Shannon entropy} of $\phi\colon X\to \Sigma$ with respect to $\mu$ given (or relative to) $\psi$ (or $\scrF$) is $H_\mu(\phi|\psi)=H_\mu(\phi|\scrF)=\displaystyle{\Ex_{y\sim \nu}[H_{\mu_y^\psi}(\phi)]}$ --- this is the amount of additional information we expect to receive when  $\phi(x)$ is revealed to us given that $x\sim \mu$ and that $\psi(x)$ was revealed to us first. 

    There are two natural notions of distance between observables. Given two observables $\phi$ and $\psi$ with the same indexing set $\Sigma$, the \emph{total variation} distance between them is  $d_{TV}(\phi,\psi)=\frac{1}{2}\sum_{a\in \Sigma} \mu(\phi^{-1}(a)\Delta \psi^{-1}(a))$, where $\Delta$ is the symmetric difference of sets --- this agrees with the usual total variation distance between $\phi_*\mu$ and $\psi_*\mu$, which measures given $x\sim \mu$, what is the probability that $\phi(x)\neq \psi(x)$. Even without $\phi$ and $\psi$ having the same indexing set, one can define the \emph{Rokhlin distance} between them by $d_{\Rok}(\phi,\psi)=H_\mu(\phi|\psi)+H_\mu(\psi|\phi)$ --- this distance measures how much additional information is gained by revealing one observable after the other.
    In addition to these distances, there is a natural order induced on the observables by \emph{refinement} and \emph{coarsening} of the partitions: An observable $\phi\colon X\to \Sigma$ is a refinement of $\psi\colon X\to \Omega$ (or alternatively, $\psi$ is a coarsening of $\phi$), denoted by $\phi\geq \psi$, if for every $a\in \Sigma$ there is a $b\in \Omega$ for which $\phi^{-1}(a)\subseteq \psi^{-1}(b)$ up to $\mu$-null sets. The \emph{join} of $\phi$ and $\psi$ is the  observable $\phi\vee\psi:X \to \Sigma\times \Omega$ given by
    $\phi \vee \psi(x)=(\phi(x),\psi(x))$.  If $\scrF_1,\scrF_2$ are two sigma-algebras then their join $\scrF_1\vee \scrF_2$ is the smallest sigma-algebra containing both; this notation is consistent with the one for observables, as $(\phi\vee \psi)^{-1}(\scrB_{\Sigma\times \Omega})=\phi^{-1}(\scrB_\Sigma)\vee \psi^{-1}(\scrB_\Omega)$. We often abuse notation and denote $\phi\vee \scrF$, by which we mean $\phi^{-1}(\scrB_\Sigma)\vee \scrF$.

     The following facts are (mostly) quite immediate

    \begin{fact}[See Lemma 2.1 in \cite{seward2016weak},  \cite{downarowicz2011entropy} and the appendix of \cite{seward2019krieger}]\label{fact:basic_prop_entropy}
        Let $(X,\scrB_X,\mu)$ be a standard probability space (equipped with a metric which makes it into a Polish space, with $\scrB_X$ the Borel sigma-algebra and $\mu$ a Borel measure), $\phi\colon X\to \Sigma$ and $\psi\colon X\to \Omega$ observables, and $\scrF,\scrF',(\scrF_n)_{n=1}^\infty$ sub sigma algebras of $\scrB_X$. Then:
        \begin{enumerate}[label=\textcolor{black}{\arabic*.}, ref=\arabic*.]
            \item \label{clause:Zero_cond_ent}\emph{Zero conditional entropy}:  $H_\mu(\phi| \scrF)=0$ if and only if $\phi^{-1}(a)\in \scrF$ up to a $\mu$-null set for every $a\in \Sigma$.
            \item \label{clause:Conditional_property}\emph{Conditional property}: $H_\mu(\phi\vee \psi|\scrF)= H_\mu(\psi|\scrF)+H_\mu(\phi|\psi\vee \scrF)$.\footnote{This is the entropic version of $\Pro[A|B]=\frac{\Pro[A\cap B]}{\Pro[B]}$.}
            \item \label{clause:Finite_coarsenings}\emph{Finite coarsenings approximate the entropy}: $H_\mu(\phi|\scrF)$ is the supremum over $H_\mu(\tau|\scrF)$  for all finite alphabet observables $\tau$ for which $\tau\leq \phi$. 
            \item \label{clause:pregiven_info}\emph{More pre-given information implies less entropy}: $\scrF'\subseteq \scrF$ up to $\mu$-null sets implies that $H_\mu(\phi|\scrF)\leq H_\mu(\phi|\scrF')$.
            \item \label{clause:limit_pregiven_info}\emph{Limit of increasing pre-given information}: If for every $n$ we have $\scrF_n\subseteq \scrF_{n+1}$, $H(\phi|\scrF_n)<\infty$ and $\scrF=\bigvee_{n}\scrF_n$, then $H(\phi|\scrF)=\inf_{n\in \mathbb{N}} H(\phi|\scrF_n)$.
            \item \label{clause:space_observable_complete_in_Rok_metric} \emph{The space of finite entropy observables is complete when equipped with the Rokhlin distance $d_\Rok$}\footnote{Here we are slightly abusing terminology. If two observables have Rokhlin distance zero then we consider them to be the same.}. 
            \item \label{clause:Distance_bounds}\emph{Distance bounds}: If $H_\mu(\phi|\scrF),H_\mu(\psi|\scrF)<\infty
            $, then $|H_\mu(\phi|\scrF)-H_\mu(\psi|\scrF)|\leq d_\Rok(\phi,\psi)$. Similarly, if $\sigma\colon X\to \Xi$ is another observable, and $H_\mu(\sigma|\phi\vee \scrF),H_\mu(\sigma|\phi\vee \scrF)<\infty$, then $|H_\mu(\sigma|\phi\vee \scrF)-H_\mu(\sigma|\phi\vee \scrF)|\leq 2d_\Rok(\phi,\psi)$. In addition, for a finite set $A\subseteq\cF$ we have $d_\Rok(\phi^A,\psi^A)\leq |A|\cdot d_\Rok(\phi,\psi)$. Finally, if $\Sigma=\Omega$, then $d_\Rok(\phi,\psi)\leq 2d_{TV}(\phi,\psi)$. 
            \item \label{clause:close_by_finite_obs}\emph{Every observable has a close by finite alphabet observable}: Let $\Sigma=\NN$, and for every $n\in \NN$ let $\phi_n(x)=\min(n,\phi(x))$. It is clear that $\phi_n\xrightarrow{TV}\phi$, and thus by \cref{clause:Distance_bounds}, $\phi_n\xrightarrow{\Rok}\phi$.
            \item \label{clause:close_by_cont_obs}\emph{Every finite alphabet observable has a close by continuous finite alphabet observable}: Let $|\Sigma|<\infty$. Suppose $X$ is a totally disconnected. By the outer regularity of $\mu$, for every $\eps>0$ there is a continuous observable $\varphi\colon X\to \Sigma$ such that $d_{TV}(\varphi,\phi)<\eps$.
            \item \label{clasues:finite_cont_obs_are_dense}\emph{The continuous finite alphabet observables are dense in all observables}: Assume $X$ is totally disconnected. From items \labelcref{clause:Distance_bounds}, \labelcref{clause:close_by_finite_obs} and \labelcref{clause:close_by_cont_obs} one can deduce that regardless of whether $\Sigma$ is finite or infinite,  for every $\eps>0$ there is a continuous finite alphabet observable $\varphi\colon X\to \Sigma'$ such that $d_\Rok(\varphi,\phi)<\eps$.
            
        \end{enumerate}
    \end{fact}
    
    \subsection{Generating observables and Rokhlin entropy}
    Let $\Gamma$ be a countable group and fix a probability measure preserving (p.m.p) action $\Gamma \cc (X,\mu)$ (on a standard probability space) and observable $\phi:X \to \Sigma$.
The group $\Gamma$ acts naturally on observables, by letting $\phi^\gamma\colon X\to \Sigma$ be the observable $\phi^\gamma(x)=\phi(\gamma^{-1}x)$ for every $\gamma\in\Gamma$. Therefore, 
a subset $S \subset \Gamma$ defines a map $ \phi^S: X \to \Sigma^S$ by letting $\phi^S=\bigvee_{s\in S} \phi^s$, or in other words
$\phi(x)(s) = \phi(s^{-1}x).$
This map arises  in two special cases: when $S$ is finite, in which case $\phi^S$ is an observable, and when $S=\G$, in which case it defines a $\G$-equivariant map from $X$ to $\Sigma^\G$. 
 A sigma algebra $\scrF\subseteq \scrB_X$ is said to be $\Gamma$-invariant if $B\in \scrF$ implies $\gamma.B\in \scrF$ for every $B\in \scrF$ and $\gamma\in \Gamma$. Similarly, a set  $B$ is $\Gamma$-invariant if $\gamma.B=B$ for every $\gamma\in \Gamma$, and we denote by $\scrI_\G$ the sigma algebra generated by $\Gamma$-invariant sets. 
 \begin{defn}[Generating observable]\label{defn:gen_ovservables}
     An observable $\phi$ is \emph{generating} with respect to a p.m.p  action $\Gamma\cc (X,\mu)$ if the smallest $\Gamma$-invariant sigma-algebra which makes $\phi$ measurable is $\scrB_X$ (up to $\mu$-null sets). Equivalently, for $\phi^{\Gamma}\colon X\to \Sigma^\Gamma$ we have that $(\phi^\Gamma)^{-1}(\scrB_{\Sigma^\Gamma})$ is $\scrB_X$ up to $\mu$-null sets.

Let $\scrF\subseteq \scrB_X$ be  a $\Gamma$-invariant sub sigma algebra. An observable $\phi$ is \emph{generating given} $\scrF$, if  $\phi^\Gamma\vee \scrF$ is $\scrB_X$ up to $\mu$-null sets
 \end{defn}
\begin{rem}\label{rem:generating}
    There is an ``operative'' way of thinking about $\phi$ being generating: If one samples $x\sim \mu$, and reveals to you $\phi^\gamma(x)$  for every $\gamma\in \Gamma$, then you can (almost surely) recover $x$. Similarly, if $\psi\colon X\to Y$ is the factor associated with $\scrF$, then $\phi$ generating given $\scrF$ implies that for $x\sim \mu$, revealing $\psi(x)$ and $\phi^\gamma(x)$  for every $\gamma\in \Gamma$ allows one to (almost surely) recover $x$.
\end{rem}

\begin{defn}[Rokhlin entropy]\label{defn:Rokhlin_entropy}
  Let $\scrF$ be a $\G$-invariant Borel sigma-algebra. The \emph{relative Rokhlin entropy} of this action given $\scrF$ is
    $$H_{\Rok}(\G \cc (X,\mu)|\scrF)=\inf H_\mu(\phi|\scrF),$$
    where the infimum is over all observables $\phi$ which generate given $\scrF$  (Definition \ref{defn:gen_ovservables}). In case $\scrF$ is trivial up to $\mu$-null sets, this quantity is called the \emph{Rokhlin entropy} and we denote it by $H_\Rok(\G\cc(X,\mu))$.
\end{defn}

\begin{rem}
    Seward usually defines the relative Rokhlin entropy  modulo the information given by $\scrI_\G$, namely as the infimum of $H_\mu(\phi|\scrF\vee \scrI_\G)$ for observables $\phi$ generating given $\scrF\vee\scrI_\G$. 
    His definition is the ``correct'' one in case of general aperiodic actions. But, in our case we mix both the aperiodic and periodic cases, and it would be more natural to use our definition. This will be consistent with the theory developed by Seward and others, since the $\scrF$ we use includes $\scrI_\G$ whenever the IRS in question is aperiodic, and in the periodic case we intentionally diverge from the standard definition to be compatible with our version of RBS (Definition \ref{defn:RBS_for_IRSs}). 
\end{rem}

    Let $\Sigma$ be a finite set, and $\kappa$  a probability measure over it. Recall that the \emph{Bernoulli shift} over $\G$ with base space $(\Sigma,\kappa)$ is the probability measure preserving action of $\G$ on the product space $(\Sigma^\G,\kappa^\G)$ given by $\gamma.c(x)=c(\gamma^{-1}x)$ for $\gamma,x\in \G$ and $c\colon \G\to \Sigma$.
    The \emph{canonical observable} of this action, which we denote by $\phi_\Sigma\colon \Sigma^\Gamma\to \Sigma$, is defined by $\phi_\Sigma(c)=c({\bf 1}_\Gamma)$. As the pushforward measure $(\phi_\Sigma)_*\kappa^\Gamma$ is $\kappa$, we have that $H_{\kappa^\Gamma}(\phi_\Sigma)=H(\kappa)$. In addition, $\phi_\Sigma$ is generating (Definition \ref{defn:gen_ovservables}), as $\phi^\gamma_\Sigma(c)=\phi_\Sigma(\gamma^{-1}.c)=\gamma^{-1}.c({\bf 1}_\Gamma)=c(\gamma)$ for every $\gamma\in\Gamma$, which means that we recover $c$ when we get access to all $\phi^\gamma_\Sigma$ (see Remark \ref{rem:generating}).    All in all, the Rokhlin entropy (Definition \ref{defn:Rokhlin_entropy}) of the Bernoulli shift satisfies $H_\Rok(\Gamma\cc (\Sigma^\Gamma,\kappa^\Gamma))\leq H(\kappa)$. A group $\Gamma$ satisfies the \emph{maximal Rokhlin entropy of Bernoulli Shifts} (RBS) criterion \cite{seward2019krieger} if for every finite $\Sigma$, and given $u$ the uniform measure over $\Sigma$, one has $H_\Rok(\Gamma\cc (\Sigma^\Gamma,u^\Gamma))= H(u)=\log|\Sigma|$.\footnote{Seward assumes this is true for \textbf{every}  standard probability  base space $(\Sigma,\kappa)$ and not only for finite ones equipped with the uniform measure. But, to deduce surjunctivity, both in his case and in the IRS case which is the topic of this paper, this weaker condition suffices.} 

    \section{\textbf{The maximal Rokhlin entropy of Bernoulli Shifts  criterion for Invariant Random Subgroups}}\label{sec:RBS}

Let $\cF$ be a finitely generated free group, and recall from Section \ref{sec:IRSs} the space $\Sub(\cF,\Sigma)$ consisting of all pairs $(K,c)$, where $K$ is a subgroup of $\cF$ and $c\colon \cF\to \Sigma$ is a $K$-invariant coloring, together with the action $w.(K,c)=(wKw^{-1},w.c)$ for $w\in \cF$. Let $\kappa$ be a probability distribution over $\Sigma$. For every $K\leq \cF$, let $\kappa_K=\kappa^{(_K\backslash^\cF)}$ be the product distribution  over right $K$-cosets of $\cF$; namely, $\kappa_K$  samples for every coset $Kw$ a color $c(Kw)\in \Sigma$ according to the distribution $\kappa$ independently of all other cosets. 
Given an invariant random subgroup (IRS) $\pi$ of $\cF$ (Definition \ref{defn:IRS}), one can define the probability measure $\kappa_\pi$ over $\Sub(\cF,\Sigma)$ to be $\Ex_{K\sim \pi}[{\bf 1}_K\times\kappa_K]$ --- in words, to sample according to $\kappa_\pi$ we need first to sample $K$ according to $\pi$, and then sample for every $Kw$ an independent color $c(Kw)$ according to $\kappa$, to retrieve $(K,c)\in \Sub(\cF,\Sigma)$. When equipping  $\Sub(\cF,\Sigma)$ with $\kappa_\pi$, the aforementioned action of $\cF$ on $\Sub(\cF,\Sigma)$ becomes p.m.p, and we call it the \emph{Bernoulli shift} over $\pi$ with base $(\Sigma,\kappa)$. As before, when $\Gamma\cong \nicefrac{\cF}{N}$ for a normal subgroup $N\trianglelefteq \cF$, the Bernoulli shift over the IRS ${\bf 1}_N$ with base $(\Sigma,\kappa)$ coincides with the Bernoulli shift over the group $\Gamma$ with the same base. Note also that, if $P\colon \Sub(\cF,\Sigma)\to \Sub(\cF)$ is the projection on the left coordinate, then the pushforward satisfies $P_*\kappa_\pi=\pi$.
We provide now the appropriate IRS analogue to the discussion on RBS from the end of Section \ref{sec:entropy_theory}. To that end we need the following simple fact:

\begin{fact}[Finite vs.\ infinite index decomposition of IRSs] \label{fact:decomposition}
Let $\pi$ be an IRS of $\cF$. There are IRSs $\pi_f$ and $\pi_\infty$ such that: $\bullet$ $\pi$ is a convex combination of them; $\bullet$ $K\sim \pi_f$ is of finite index; $\bullet$ $K\sim \pi_\infty$ is of infinite index. 

We call $\pi_f$ the \emph{finite component} of $\pi$ and $\pi_\infty$ its \emph{infinite component}. 
\end{fact}

\begin{rem}\label{rem:more_on_the_decomposition}
    Let $\Sub_f(\cF)$ be the set of finite index subgroups of $\cF$ and $\Sub_\infty(\cF)$ the set of infinite index ones. As the singleton $\{K\}$ for every $K\leq \cF$ with $[\cF:K]<\infty$ is open, $\Sub_f(\cF)$ is an open set, and thus $\Sub_\infty(\cF)$ is closed. The above condition ``a subgroup $K$ sampled according to $\pi_f$ is of finite index'' is equivalent to $\pi_f$ being a convex combination of finitely described IRSs \eqref{eq:fin_desc_IRS}, while the other condition ``a subgroup $K$ sampled according to $\pi_\infty$ is of infinite index'' is the same as $\supp(\pi_\infty)\subseteq \Sub_\infty(\cF)$.

    We denote by $\Sub_f(\cF,\Sigma)$ all $(H,c)\in \Sub(\cF,\Sigma)$ such that $H\in \Sub_f(\cF)$, and  $\Sub_\infty(\cF,\Sigma)$ is similarly defined.  
\end{rem}

\begin{defn}\label{defn:canonical_obs}
The \emph{canonical observable} $\phi_\Sigma\colon \Sub(\cF,\Sigma)\to \Sigma$ is defined to be $\phi_\Sigma(K,c)=c(K)$.  
\end{defn}

\begin{claim}
    Let $\pi$ be an IRS of $\cF$, and $(\Sigma,\kappa)$ a finite probability space. Then, the canonical  observable $\phi_\Sigma$ is generating (Definition \ref{defn:gen_ovservables}) for the Bernoulli shift over $\pi$ with base $(\Sigma,\kappa)$ \textbf{given} $P$. In addition, $H_{\kappa_\pi}(\phi_\Sigma|P)=H(\kappa)$, which in turn implies  $H_\Rok(\cF\cc (\Sub(\cF,\Sigma),\kappa_\pi)|P)\leq H(\kappa)$.
\end{claim}
\begin{proof}
    For the generation claim, use the operative framework of Remark \ref{rem:generating}. If $(K,c)\sim \kappa_\pi$, and we get access to $P(K,c)=K$ as well as 
    \begin{equation}\label{eq:canonical_generates}
\phi^w_\Sigma(K,c)=\phi_\Sigma(w^{-1}.(K,c))=\phi_\Sigma(w^{-1}Kw,w^{-1}.c)=w^{-1}.c(w^{-1}Kw)=c(Kw)
    \end{equation}
    for every $w\in \cF$, then we  recover $(K,c)$ completely. Hence $\phi_\Sigma$ is generating given $P$.  

    For the entropy condition,  as $P_*\kappa_\pi=\pi$, and the conditional distributions (Section \ref{sec:basic_notions}) satisfy $(\kappa_\pi)^P_K={\bf 1}_K\times \kappa_K$ for every $K\leq \cF$, we have that
\begin{equation}\label{eq:entropy_of_canonical_obs}
    H_{\kappa_\pi}(\phi_\Sigma|P)=\Ex_{K\sim \pi}[H_{{\bf 1}_K\times \kappa_K}(\phi_\Sigma)]=\Ex_{K\sim \pi}[H(\underbrace{(\phi_\Sigma)_*{\bf 1}_K\times \kappa_K)}_{=\kappa})]=H(\kappa).
\end{equation}
    As the relative Rokhlin entropy (Definition \ref{defn:Rokhlin_entropy}) $H_\Rok(\cF\cc (\Sub(\cF,\Sigma),\kappa_\pi)|P)$ is the infimum of $H_{\kappa_\pi}(\phi|P)$ over all observables which generate given $P$, we are done.
\end{proof}

\begin{defn}[RBS]\label{defn:RBS_for_IRSs}
    Let $\pi$ be an IRS \textbf{supported on infinite index subgroups}. We say that $\pi$ satisfies the Rokhlin entropy of Bernoulli Shifts (RBS) criterion if $H_\Rok(\cF\cc (\Sub(\cF,\Sigma),u_\pi)|P)= H(u)=\log|\Sigma|$ for every finite base $\Sigma$ equipped with the uniform probability distribution $u$.
    We will say that a general IRS $\pi$ (not necessarily supported on infinite index subgroups) satisfies RBS if its infinite component $\pi_\infty$ (as in Fact \ref{fact:decomposition}) does. In particular, by definition, if $\pi$-a.e. subgroup has finite-index, then $\pi$ satisfies RBS.     We denote the set of  IRSs satisfying RBS by $\IRS_\RBS(\cF)$.
\end{defn}
By Definition \ref{defn:RBS_for_IRSs},  every IRS which is supported on finite index subgroups satisfies RBS. But, it is not clear from the definition that $\IRS_\RBS(\cF)$ is a closed subset of $\IRS(\cF)$. Hence, one needs still to argue why every cosofic IRS satisfies RBS.
In the next section, we provide a condition equivalent to RBS which will allow us both to construct an outer approximation algorithm for $\IRS_{\RBS}(\cF)$, as well as to show that $\IRS_\RBS(\cF)$ is weak* closed (and thus every cosofic IRS will satisfy RBS).  Let us finish this section by proving:
\begin{thm}\label{thm:RBS_implies_surj}
    If an IRS satisfies the RBS criterion (Definition \ref{defn:RBS_for_IRSs}), then it is surjunctive (Definition \ref{defn:surjunctive}).
\end{thm}

To that end, we need to collect additional definitions  and claims. The following fact is straightforward  and we leave it without a proof:

\begin{fact}[Basic surjunctiviy properties]\label{fact:props_of_IRSs_surj}
Let $\pi,\pi_1$ and $\pi_2$ be IRSs of $\cF$. We have:
\begin{enumerate}
    \item \emph{Surjunctivity is preserved under convex combinations}:  If $\pi_1$ and $\pi_2$ are  surjunctive (Definition \ref{defn:surjunctive}), then  every convex combination of $\pi_1$ and $\pi_2$ is also surjunctive.
    \item \emph{Finite index IRSs are surjunctive}: If $\pi$ is supported on finite index subgroups, then it is surjunctive.
\end{enumerate}
\end{fact}

Let $([0,1],\lambda)$ be the interval equipped with the Lebesgue measure on it. 
For a subgroup $K\leq \cF$, let $\lambda_K=\lambda^{(_K\backslash^\cF)}$ be the distribution that associates to each right $K$-coset an independent uniform real number between $0$ and $1$, which we think of as the ``time in which this coset appears''. Almost surely, $t\colon _K\backslash^\cF\to [0,1]$ sampled according to $\lambda_K$ induces a linear order on $_K\backslash^\cF$, and we let $L_{K,t}=\{w\in \cF \mid t(K)>t(Kw)\}$ to consist of all ``{past} group elements''.
Given an IRS $\pi$, let $\lambda_\pi=\Ex_{K\sim \pi}[{\bf 1}_K\times \lambda_K]$, namely sample a subgroup $K$ according to $\pi$ and then sample a uniform real number between $0$ and $1$ for each $K$-coset independently. 

Every p.m.p action $\cF\cc (X,\mu)$ induces an IRS $\Stab(\mu)$ of $\cF$ by letting $K=\Stab(x)$ for $x\sim \mu$. Given an IRS $\pi$, we say that $\cF\cc (X,\mu)$ has stabilizer type $\pi$, if $\Stab(\mu)=\pi$. Let $\mu_K=\mu^{\Stab}_K$ be the associated conditional distributions (Section \ref{sec:basic_notions}).

\begin{defn}[Percolation entropy]
    Let  $\pi$ be an  IRS of $\cF$,  $\cF\cc (X,\mu)$ a p.m.p action with stabilizer type $\pi$,   $\phi\colon X\to \Omega$ an observable (Section \ref{sec:observables}) and $\scrF$ an $\cF$-invariant sub sigma algebra of $\scrB_X$. The  percolation entropy of $\phi$ given $\scrF$  is
    \begin{equation}\label{eq:def_percolation}
H_{\perc,\mu}(\phi|\scrF)=\Ex_{(K,t)\sim \lambda_\pi}[H_{\mu_K}(\phi|\phi^{L_{K,t}}\vee \scrF)].
    \end{equation}
\end{defn}

\begin{rem}
    Let us unpack the above definition, as it is a bit mysterious. By the definition of conditional distributions, and as $\pi$ is the stabilizer type of $\mu$, sampling $x\sim \mu$ is the same as first sampling $K\sim \pi$ and then $x\sim \mu_K$. On the right hand side of \eqref{eq:def_percolation}, we have one more object sampled, which is a time function $t\colon _K\backslash^\cF\to [0,1]$ sampled according to $\lambda_K$. Let $\psi$ be the factor associated  to $\scrF$. Then, the quantity $H_{\mu_K}(\phi|\phi^{L_{K,t}}\vee\scrF)$ measures how much information is gained by learning $\phi(x)$ for $x\sim \mu_K$, given that you have learned the value $\psi(x)$ and the conjugates $\phi^w(x)$ for every past element $w\in L_{K,t}$. Namely,  given all evaluations of past conjugates of $\phi$ (in addition to $\psi$), how much more information will the evaluation of $\phi$ provide?
    The next claim contains the main example we will need, which will also shed some light on this definition.
\end{rem}

\begin{claim}\label{claim:props_of_perc_entropy}
Let $\pi$ be an IRS  of $\cF$ supported on infinite index subgroups,   $\Sigma$ a finite set, $|\Sigma|>1$, and $P\colon \Sub(\cF,\Sigma)\to \Sub(\cF)$ the projection on the first coordinate. Recall the shift action $w.(K,c)=(wKw^{-1},w.c)$ of $\cF$ on $\Sub(\cF,\Sigma)$. 
\begin{enumerate}
    \item Let $u$ be the uniform measure over $\Sigma$, $\cF\cc(\Sub(\cF,\Sigma),u_\pi)$ the associated Bernoulli shift, and $\phi_\Sigma\colon \Sub(\cF,\Sigma)\to \Sigma$  the canonical observable $\phi_\Sigma(K,c)=c(K)$. Then, the stabilizer type of $u_\pi$ is $\pi$, the sigma algebra of $\cF$-invariant sets $\scrI_\cF$ is contained in $P^{-1}(\scrB_{\Sub(\cF)})$ up to $u_\pi$-null sets,  and $H_{\perc,u_\pi}(\phi_\Sigma|P)=\log|\Sigma|$.
    \item Let $\cF$ still act on $\Sub(\cF,\Sigma)$ by  shifts,  and let $\mu$ be  a Borel probability measure on $\Sub(\cF,\Sigma)$ which is preserved under this action. Assume the stabilizer type of $\mu$ is $\pi=P_*\mu$  and $H_{\perc,\mu}(\phi_\Sigma|P)=\log|\Sigma|$. Then $\mu=u_\pi$. 
\end{enumerate} 
\end{claim}

\begin{proof}

{\bf Claim 1}.  The stabilizer type of $u_\pi$ is $\pi$.

\begin{proof}[Proof of Claim 1]
Given a subgroup $K \le \cF$, recall that  $u_K=u^{_K \backslash ^\cF}$ is the product measure on $\Sigma^{_K \backslash ^\cF}$, that ${\bf 1}_K$ is the Dirac measure concentrated on $K$, and that
$u_\pi = \Ex_{K\sim \pi}[ {\bf 1}_K\times u_K].$   
The stabilizer of $K\leq \cF$ with respect to conjugation is the normalizer ${\rm Nor}(K)$, and the stabilizer of $c\colon \cF\to K$ with respect to the shift action is $M(c)=\{w\in \cF \mid \forall x\in \cF\colon c(x)=c(w^{-1}x)\}$. Hence, the stabilizer of $(K,c)$ is ${\rm Nor}(K)\cap M(c)$. As $c$ is  $K$-invariant, $K\subseteq M(c)$ and thus $K\subseteq \Stab(K,c)$ for every $(K,c)\in \Sub(\cF,\Sigma)$. By Boole's inequality,
\[
\begin{split}
\Pro_{(K,c)\sim u_\pi}[K\neq \Stab(K,c)]&=\Pro_{(K,c)\sim u_\pi}[\exists w\in \cF\colon w\notin K \land w\in \Stab(K,c)]\\
&\leq \sum_{w\in \cF} \Pro_{(K,c)\sim u_\pi}[w\notin K\land w\in \Stab(K,c)]\\
&=(*).
\end{split}
\]
   Let ${\rm rad}(\pi)=\displaystyle{\bigcap_{K\in \supp(\pi)}}K$. Then, for every $w\notin {\rm rad}(\pi)$, there is a positive probability that $K\sim \pi$ satisfies $w\notin K$. Hence,
   \[
(*)=\sum_{w\notin {\rm rad}(\pi)}\Pro_{(K,c)\sim u_\pi}[w\notin K]\cdot \Pro[w\in \Stab(K,c)\mid w\notin K].
   \]
    Hence, if we can show that 
    $\Pro[w\in \Stab(K,c)\mid w\notin K]=0$ for every infinite index subgroup $K$, then we are done. 
   Because $K$ has infinite index in $\cF$, there is an infinite $I \subset \cF$ such that  $\Theta=\{K\gamma, Kw^{-1}\gamma\mid \gamma\in I\}$ are all different $K$-cosets. As $c$ was sampled from $u_K$, the color of each coset in $\Theta$ is independent of every other coset. In particular,  $\Pro[c(\gamma)=c(w^{-1}\gamma)]=\frac{1}{|\Sigma|}<1$. Therefore, with probability $1$ there is a $\gamma\in I$ for which $c(\gamma)\neq c(w^{-1}\gamma)=w.c(\gamma)$, and $w\notin \Stab(K,c)$.
    
\end{proof}

{\bf Claim 2}. The sigma algebra of $\cF$-invariant sets $\scrI_\cF$ is contained in $P^{-1}(\scrB_{\Sub(\cF)})$ up to $u_\pi$-null sets.

\begin{proof}[Proof of Claim 2]
Recall that $\Sub_\infty(\cF)$ (Remark \ref{rem:more_on_the_decomposition}) is the closed set of infinite index subgroups of $\cF$. 
There is a sequence of continuous maps $\phi_n\colon \Sub_\infty(\cF)\to \cF$ such that for every $K\in \Sub_\infty(\cF)$ and for any finite set $W \subseteq \cF$, there  exists $n_0$ such that $\phi_n(K)^{-1}\notin \bigcup_{w,w'\in W} Kww'^{-1}$ for all $n\geq n_0$ --- for example, $\phi_n$ can lexicographically order the sphere of radius $n$ in the Schreier graph $\Sch(\cF,K,S)$ and output the shortest representative of the first vertex in this order.\footnote{These functions are continuous as the Schreier graph up to radius $n$ can be constructed by only knowing the intersection of $K$ with a ball of radius $2n+1$.}

For $f \in L^1(\Sub(\cF,\Sigma),u_\pi)$ and $K\in \supp(\pi)\subseteq \Sub_\infty(\cF)$, let $\widetilde f(K)=\displaystyle{\Ex_{c\sim u_K}}[f(K,c)]$ and let $f\circ \phi_n(K,c)=f(\phi_n(K).(K,c))$. We claim the following mixing property: For every  $f_1,f_2\in L^1(\Sub(\cF,\Sigma),u_\pi)$,  the covariance of $f_1\circ\phi_n$ and $f_2$ given $P$ tends to zero, namely
\begin{align}\label{E:claim2}
   \lim_{n\to\infty} \Ex_{(K,c)\sim u_\pi}[ f_1\circ \phi_n(K,c)\cdot  f_2(K,c) -\widetilde {f_1\circ \phi_n}(K)\cdot \widetilde f_2(K)]=0.
\end{align}
In order to prove this statement, it suffices to prove it in the special case in which both $f_1,f_2$ are continuous, because continuous functions are dense in $L^1(\Sub(\cF,\Sigma),u_\pi)$. 
Assume $f_1,f_2$ are continuous. Then there exists a finite set $W \subset \cF$ such that  $f_1(K,c)$ and $f_2(K,c)$  depend only on $K\cap W$ and the restriction of $c$ to $W$. For a large enough $n$ we have $\phi_n(K)^{-1}\notin \bigcup_{w,w'\in W} Kww'^{-1}$, which implies in particular that the cosets $\{Kw\}_{w\in W}$ and $\{K\phi_n(K)^{-1}w\}_{w\in W}$ are disjoint.   As $u_K$ colors each coset independently, we deduce that $c|_W$ and $\phi_n(K).c|_W=c|_{\phi_n(K)^{-1}\cdot W}$ are independent when $c\sim u_K$. Thus,  $\Ex_{(K,c)\sim u_\pi}[ f_1\circ \phi_n(K,c)\cdot  f_2(K,c) -\widetilde {f_1\circ \phi_n}(K)\cdot \widetilde f_2(K)]=0$, which proves \eqref{E:claim2}.

Now let $E \subseteq \Sub(\cF,\Sigma)$ be a $\cF$-invariant measurable subset. Let $f_1={\bf 1}_E$ and $f_2=1-f_1$. As $E$ is $\cF$-invariant,  $f_1=f_1\circ \phi_n$ for every $n$, and we have $(f_1\circ \phi_n)\cdot  f_2 =0$ almost everywhere. So \eqref{E:claim2} implies that for $\pi$-almost every $K$,  $\widetilde f_1 (K)=0\  {\rm or}\ 1$ --- in other words, for $\pi$-almost every $K$, either for $u_K$-almost all $c\colon _K\backslash^\cF\to \Sigma$  we have $(K,c)\in E$, or for almost none of them. This implies that $E$ is in $P^{-1}(\mathscr{B}_{\Sub(\cF)})$ up to $u_\pi$-null sets, as required.


\end{proof}

{\bf Claim 3}. $H_{\perc,u_\pi}(\phi_\Sigma|P)=\log|\Sigma|$.

\begin{proof}[Proof of Claim 3]
By definition,
$$H_{\perc,u_\pi}(\phi_\Sigma|P)=\Ex_{(K,t)\sim \lambda_\pi}[H_{{\bf 1}_K\times u_K}(\phi_\Sigma|\phi_\Sigma^{L_{K,t}}\vee P)].$$
When sampling $(K,c)\sim {\bf 1}_K\times u_K$, $P(K,c)=K$ gives us no new information, and  $H_{{\bf 1}_K\times u_K}(\phi_\Sigma|\phi_\Sigma^{L_{K,t}}\vee P)=H_{{\bf 1}_K\times u_K}(\phi_\Sigma|\phi_\Sigma^{L_{K,t}})$. In addition, as $\Id_\cF\notin L_{K,t}$, $\phi_\Sigma$ is independent of $\phi_\Sigma^{L_{K,t}}$ and $$H_{{\bf 1}_K\times u_K}(\phi_\Sigma|\phi_\Sigma^{L_{K,t}})=H_{{\bf 1}_K\times u_K}(\phi_\Sigma)=H(\phi_{\Sigma*}({\bf 1}_K\times u_K))=H(u)=\log|\Sigma|. \qedhere$$ 
\end{proof}

{\bf Claim 4}. Assume $H_{\perc,\mu}(\phi_\Sigma|P)=\log|\Sigma|$. Then $\mu=u_\pi$.

\begin{proof}[Proof of Claim 4]
As $\mu$ both has  stabilizer type $\pi$ and $P_*\mu=\pi$, we have
$$H_{\perc,\mu}(\phi_\Sigma|P)=\Ex_{(K,t)\sim \lambda_\pi}[H_{\mu_K}(\phi_\Sigma|\phi_\Sigma^{L_{K,t}}\vee P)],$$
where  $\mu=\Ex_{K\sim \pi}[{\bf 1}_K\times \mu_K]$. 
Because $\phi_\Sigma$ takes values in $\Sigma$,
$H_{\mu_K}(\phi_\Sigma|\phi_\Sigma^{L_{K,t}}\vee P)\le \log |\Sigma|$,
and equality holds if and only if $\phi_\Sigma$ is independent of $\phi_\Sigma^{L_{K,t}}\vee P$ (for a.e.  $(K,t)$) and $\phi_{\Sigma *}\mu_K$ is the uniform measure on  $\Sigma$. But this implies that if $c \sim \mu_K$ is a random sample, then $c(K)$ is independent of $c(Kw)$ for all non-identity cosets $Kw$ (since each coset $Kw$ appears in $L_{K,t}$ with positive probability). Moreover, the variables $c(Kw)$ are uniformly distributed on $\Sigma$. Thus $\mu_K$ is $u_K$ and $\mu=\pi$.
\end{proof}
\renewcommand{\qedsymbol}{}
\end{proof}

The reason we have defined percolation entropy is because of the following result:

\begin{fact}[Corollary 7.3 in \cite{seward2016weak}]\label{fact:cor7.3Seward}
        Let  $\pi$ be an IRS of $\cF$  supported on infinite index subgroups,  $\cF\cc (X,\mu)$  a p.m.p action with stabilizer type $\pi$,   $\phi\colon X\to \Omega$ a \textbf{generating} observable (Section \ref{sec:observables} and Definition \ref{defn:gen_ovservables}) and $\scrF$ an $\cF$-invariant sub sigma algebra of $\scrB_X$ which contains the sub algebra $\scrI_\cF$ of $\cF$-invariant sets. Then
        \[
H_\Rok(\cF\cc(X,\mu)|\scrF)\leq H_{\perc,\mu}(\phi|\scrF).
        \]
\end{fact}

\begin{proof}[Proof of Theorem \ref{thm:RBS_implies_surj}]
    By Fact \ref{fact:props_of_IRSs_surj}, it is enough to prove the theorem assuming $\pi$ is supported on infinite index subgroups. Let $\Phi\colon \dom(\pi)\to \dom(\pi)$ be an injective cellular automaton. Since $u_\pi$ is supported on $\dom(\pi)$, we can push it forward along $\Phi$ and get $\mu=\Phi_*u_\pi$. As $\Phi$ is $\cF$-equivariant,  we have 
    \[
    \forall (K,c)\in \Sub(\cF,\Sigma),w\in \Stab(\Phi(K,c))\ \colon \ \ \Phi(K,c)=w.\Phi(K,c)=\Phi(w.(K,c)),
    \]
    but it is assumed to be injective and thus $(K,c)=w.(K,c)$ and $w\in \Stab(K,c)$. Hence, $\Stab(\Phi(K,c))=\Stab(K,c)$ and $\mu$ has the same stabilizer type as $u_\pi$, which is $\pi$ by Claim \ref{claim:props_of_perc_entropy}.
    Now, the shift $\cF\cc(\Sub(\cF,\Sigma),\mu)$ is a p.m.p action which is isomorphic to the original Bernoulli shift (as $\Phi$ is injective). Since $\pi$ satisfies RBS (Definition \ref{defn:RBS_for_IRSs}), we have 
    \[
H_{\Rok}(\cF\cc(\Sub(\cF,\Sigma),\mu)|P)=H_{\Rok}(\cF\cc(\Sub(\cF,\Sigma),u_\pi)|P)=\log|\Sigma|.
    \]
    By Claim \ref{claim:props_of_perc_entropy},  $\scrI_\cF\subseteq P^{-1}(\scrB_{\Sub(\cF)})$, and thus by Fact \ref{fact:cor7.3Seward}, we get that $H_{\perc,\mu}(\phi_\Sigma|P)\geq \log|\Sigma|$. As  
    \[
    H_{\mu_K}(\phi_\Sigma|\phi_\Sigma^{L_{K,t}}\vee P)\leq H_{\mu_K}(\phi_\Sigma)=H((\phi_{\Sigma})_*\mu_K)\leq \log|\Sigma|
    \]
    (where the first inequality is due \cref{clause:pregiven_info} in Fact \ref{fact:basic_prop_entropy} and the last is because $(\phi_{\Sigma})_*\mu_K$ is some probability distribution over $\Sigma$), we deduce that $H_{\perc,\mu}(\phi_\Sigma|P)$ is also upper bounded by $\log|\Sigma|$. Hence, $H_{\perc,\mu}(\phi_\Sigma|P)=\log|\Sigma|$, and by $(2)$ in Claim \ref{claim:props_of_perc_entropy} we have $\mu=u_\pi$. Since $\Phi$ is continuous, its image agrees with $\supp(\mu)=\supp (u_\pi)=\dom(\pi)$ and it is thus  surjective. 
\end{proof}

\section{\textbf{A  condition equivalent to RBS}}\label{sec:cont_cond}
In this section we prove a condition equivalent to RBS. In some sense, this result is the key for providing an outer approximation (Theorem \ref{thm:outer_approximation}) and thus for proving Theorem \ref{thm:main}. This is because it translates the RBS condition, which is about measurable (which is a non-finitary condition) generating  (a global condition) observables, to a condition for general continuous observables, allowing local analysis.
\begin{thm}\label{thm:equiv_cond_to_RBS_via_cont_obs}
    An IRS $\cF$ satisfies RBS (Definition \ref{defn:RBS_for_IRSs}) if and only if for every pair of finite sets $\Sigma,\Omega$  and every continuous observable $\psi\colon \Sub(\cF,\Sigma)\to \Omega$ we have
    \begin{equation}\label{eq:cont_condition}
        \log|\Sigma|=H(u)\leq H_{u_\pi}(\psi|P)+H_{u_\pi}(\phi_\Sigma|\psi^\cF\vee P),
    \end{equation}
    where $u$ is the uniform measure over $\Sigma$ and $\phi_\Sigma$ is the canonical observable (Definition \ref{defn:canonical_obs}).
\end{thm}
\begin{proof}
    Let us first prove that satisfying condition \eqref{eq:cont_condition} for every continuous observable implies the RBS criterion. Assume $\pi$ does not satisfy RBS, and let $\pi=\alpha_\infty\pi_\infty+\alpha_f\pi_f$ be its decomposition to infinite and finite components as in Fact \ref{fact:decomposition}. Then, $\alpha_\infty>0$, and there is a  (measurable) observable $\psi\colon \Sub(\cF,\Sigma)\to \Omega$ which is generating given $P$, namely $\psi^\cF\vee P$ is the Borel sigma algebra of $\Sub(\cF,\Sigma)$ up to $\pi_\infty$-null sets, and also $H_{u_{\pi_\infty}}(\psi|P)<\log|\Sigma|$. Let us define a new observable $\psi'\colon \Sub(\cF,\Sigma)\to \Omega\sqcup \Sigma$:
    \[
\psi'(K,c)=\begin{cases}
    \phi_\Sigma(K,c) & [\cF:K]<\infty,\\
    \psi(K,c) & [\cF:K]=\infty,
\end{cases}
    \]
    where $\phi_\Sigma$ is the canonical observable (Definition \ref{defn:canonical_obs}). We claim that: $\bullet$ $\psi'$ is measurable, $\bullet$ it is  generating given $P$ up to $\pi$-null sets,  $\bullet$ $H_{u_\pi}(\psi|P)<\log|\Sigma|$. To that end, recall the notations $\Sub_f(\cF)$ and $\Sub_\infty(\cF)$ (Remark \ref{rem:more_on_the_decomposition}) for the finite and infinite index subgroups of $\cF$ respectively. 
    The inverse image of $a\in \Sigma$ has the form $\phi^{-1}_\Sigma(a)\cap \Sub_f(\cF,\Sigma)$ which is the intersection of two open subsets of $\Sub(\cF,\Sigma)$ and is therefore open. The inverse image of $b\in \Omega$ is $\psi^{-1}(b)\cap \Sub_\infty(\cF,\Sigma)$ which is the intersection of a Borel subset and a closed subset and is therefore Borel. All in all, $\psi'$ is measurable. For generation, let $(K,c)\sim u_\pi$, and assume we are given the information $\psi'^w(K,c)$ for every $w\in \cF$ and $P(K,c)=K$. In case $K$ is of finite index, then $\psi'^w(K,c)=c(Kw)$ \eqref{eq:canonical_generates} which determines $c$ completely. In case $K$ is of infinite index, then it is in the support of $\pi_\infty$ and by our assumption that $\psi$ generates with respect to 
    $\pi_\infty$, as $\psi'^w(K,c)=\psi^w(K,c)$ in this case, $c$ is completely determined. Hence, by Remark \ref{rem:generating}, $\psi'$ is generating. 
    Finally, as $\psi'$ is generating, we have $H_{u_\pi}(\phi_\Sigma|\psi'^\cF\vee P)=0$ by \cref{clause:Zero_cond_ent} of Fact \ref{fact:basic_prop_entropy}. Moreover, by the definition of conditional entropy (Section \ref{sec:observables}), we have
    \begin{equation}\label{eq:psi'_is_refuting_4.1}
        H_{u_\pi}(\psi'|P)=\alpha_f \underbrace{H_{u_{\pi_f}}(\phi_\Sigma|P)}_{=_{\eqref{eq:entropy_of_canonical_obs}}\log|\Sigma|}+\alpha_\infty \underbrace{H_{u_{\pi_\infty}}(\psi|P)}_{<\log|\Sigma|}<\log|\Sigma|.
    \end{equation}
    It may seem that we are done with this direction, as $\psi'$ does not satisfy \eqref{eq:cont_condition}, but $\psi'$ is only measurable and not continuous. Let us amend this. Since $H_{u_\pi}(\phi_\Sigma|\psi'^\cF\vee P)=0$, by \cref{clause:limit_pregiven_info} in Fact \ref{fact:basic_prop_entropy}, for every $\eps>0$ there is some finite subset $A\subseteq\cF$ for which $H_{u_\pi}(\phi_\Sigma|\psi'^A\vee P)<\eps$. By \cref{clause:Distance_bounds} in  Fact \ref{fact:basic_prop_entropy}, if $\psi''$ is an observable which satisfies $d_\Rok(\psi'',\psi')\leq\delta$, then $d_\Rok(\psi''^A,\psi'^A)\leq |A|\delta$ and 
    \[
    \begin{split}
    &|H_{u_\pi}(\psi''|P)-H_{u_\pi}(\psi'|P)|\leq d_\Rok(\psi'',\psi')\leq \delta, \\
    &|H_{u_\pi}(\phi_\Sigma|\psi''^A\vee P)-H_{u_\pi}(\phi_\Sigma|\psi'^A\vee P)|\leq 2d_\Rok(\psi''^A,\psi'^A)\leq 2|A|\delta.
    \end{split}
    \]
Hence, 
\[
\begin{split}
    H_{u_\pi}(\psi''|P)+H_{u_\pi}(\phi_\Sigma|\psi''^A\vee P)&\leq H_{u_\pi}(\psi'|P)+\delta+H_{u_\pi}(\phi_\Sigma|\psi'^A\vee P)+2|A|\delta\\
    &\leq H_{u_\pi}(\psi'|P)+(2|A|+1)\delta+\eps.
\end{split}
\]
As $H_{u_\pi}(\psi'|P)<\log|\Sigma|$ \eqref{eq:psi'_is_refuting_4.1}, we can choose $\eps$ and then $\delta$  so that $ H_{u_\pi}(\psi''|P)+H_{u_\pi}(\phi_\Sigma|\psi''^A\vee P)<\log|\Sigma|$. By \cref{clasues:finite_cont_obs_are_dense} in Fact \ref{fact:basic_prop_entropy}, we can choose $\psi''$ to be continuous and are thus done with this direction.

 Since
\[
\begin{split}
 H_{u_\pi}(\psi|P)+H_{u_\pi}(\phi_\Sigma|\psi^\cF\vee P)=&\alpha_\infty  (H_{u_{\pi_\infty}}(\psi|P)+H_{u_{\pi_\infty}}(\phi_\Sigma|\psi^\cF\vee P))\\
 +&\alpha_f  (H_{u_{\pi_f}}(\psi|P)+H_{u_{\pi_f}}(\phi_\Sigma|\psi^\cF\vee P)),
\end{split}
\]
where $\pi=\alpha_\infty\pi_\infty +\alpha_f\pi_f$ is the decomposition as in Fact \ref{fact:decomposition}, we can prove that RBS implies \eqref{eq:cont_condition} separately for the finite index and infinite index cases.

Assume $\pi$ samples only  finite index subgroups (so it satisfies RBS by Definition \ref{defn:RBS_for_IRSs}). Then, for every observable $\psi\colon \Sub(\cF,\Sigma)\to \Omega$ and finite index subgroup $K\leq \cF$ we have
\begin{equation}\label{eq:proof_of_cont_condition_for_finite_index}
\begin{split}
\log(\Sigma)&=\frac{[\cF:K]\cdot\log(\Sigma)}{[\cF:K]}\\
&=\frac{H_{u_K}(\phi^{\cF}_\Sigma)}{[\cF:K]}\\
&=\frac{H_{u_K}(\phi^{\cF}_\Sigma\vee \psi^\cF)}{[\cF:K]}\\
&=\frac{H_{u_K}(\psi^\cF)+H_{u_K}(\phi_\Sigma^\cF|\psi^\cF)}{[\cF:K]}\\
&\leq H_{u_K}(\psi)+H_{u_K}(\phi_\Sigma|\psi^\cF)
\end{split}
\end{equation}
where the second equality is since $\phi^\cF_{\Sigma*}u_K=u^{[\cF:K]}$, the third equality is since $\phi^\cF_\Sigma$ was already providing all information, the fourth equality is from the conditional property in \cref{clause:Conditional_property} of Fact \ref{fact:basic_prop_entropy}, and the last inequality is a combination of the conditional property and \cref{clause:pregiven_info} in Fact \ref{fact:basic_prop_entropy}. 
As $H_{u_\pi}(\psi|P)+H_{u_\pi}(\phi_\Sigma|\psi^\cF\vee P)=\Ex_{K\sim \pi}[H_{u_K}(\psi)+H_{u_K}(\phi_\Sigma|\psi^\cF)]$, and $\pi$ was assumed to sample only finite index subgroups, this case is clear from \eqref{eq:proof_of_cont_condition_for_finite_index}.

Assume now that $\pi$ is supported on infinite index subgroups and satisfies RBS, and that $\psi\colon \Sub(\cF,\Sigma)\to \Omega$ is an observable. As $\phi_\Sigma$ is generating relative to $P$, it is also generating given $P$ and $\psi^\cF$. By the definition of Rokhlin entropy (Definition \ref{defn:Rokhlin_entropy}), we thus have
\begin{equation}\label{eq:canonical_is_gen_given_psi_and_P}
H_\Rok(\cF\cc(\Sub(\cF,\Sigma),u_{\pi})|\psi^\cF\vee P)\leq H_{u_\pi}(\phi_\Sigma|\psi^\cF\vee P).\end{equation}
Since $\pi$ is supported on infinite index subgroups, and the p.m.p action $(\cF\cc(\Sub(\cF,\Sigma),u_{\pi})$ has stabilizer type $\pi$ (Claim \ref{claim:props_of_perc_entropy}), it is in particular aperiodic. Moreover, the algebra $(\psi^\cF)^{-1}(\scrB_{\Omega^\cF}) \vee P^{-1}(\scrB_{\Sub(\cF)})$ is $\cF$-invariant. Corollary 1.4 in \cite{MR4308159} states that in such a case, for every $\eps>0$ there is an observable $\alpha\colon \Sub(\cF,\Sigma)\to \Xi$ which generates this action given $\psi^\cF\vee P$, namely $(\psi^\cF)^{-1}(\scrB_{\Omega^\cF}) \vee P^{-1}(\scrB_{\Sub(\cF)})\vee (\alpha^{\cF})^{-1}(\scrB_{\Xi^\cF})$ is $\scrB_{\Sub(\cF,\Sigma)}$ up to $u_\pi$-null sets,  and 
\begin{equation}\label{eq:the_Alpeev_Seward_observable}
    H_{u_\pi}(\alpha)\leq H_\Rok(\cF\cc(\Sub(\cF,\Sigma),u_{\pi})|\psi^\cF\vee P)+\eps.
\end{equation}
Note that this is somewhat surprising, and is indeed a non trivial fact. From Definition \ref{defn:Rokhlin_entropy} it is clear that there is such an $\alpha$ for which $H(\alpha|\psi^\cF\vee P)$ is bounded from above by the Rokhlin entropy plus $\eps$, but the entropy of $\alpha$ itself without access to $\psi^\cF\vee P$ is potentially much higher. Nevertheless, Alpeev and Seward showed that there is such an observable satisfying \eqref{eq:the_Alpeev_Seward_observable}. Note that $\alpha$ generating given $\psi^\cF\vee P$ is the same as $\psi \vee\alpha$ generating given $P$. Therefore,
\[
\begin{split}
H_\Rok(\cF\cc(\Sub(\cF,\Sigma),u_{\pi})| P)&\leq H_{u_{\pi}}(\psi\vee\alpha|P)\\
&= H_{u_\pi}(\psi|P)+\underbrace{H_{u_\pi}(\alpha|\psi\vee P)}_{\leq H_{u_\pi}(\alpha)}\\
&\leq_{\eqref{eq:the_Alpeev_Seward_observable}} H_{u_\pi}(\psi|P)+H_\Rok(\cF\cc(\Sub(\cF,\Sigma),u_{\pi})|\psi^\cF\vee P)+\eps\\
&\leq_{\eqref{eq:canonical_is_gen_given_psi_and_P}}  H_{u_\pi}(\psi|P)+H_{u_\pi}(\phi_\Sigma|\psi^\cF\vee P)+\eps.
\end{split}
\]
Since this is true for every $\eps>0$,  and as we assumed $\pi$ satisfies RBS --- namely that $H_\Rok(\cF\cc(\Sub(\cF,\Sigma),u_{\pi})| P)=\log|\Sigma|$ --- we deduce \eqref{eq:cont_condition} is satisfied in this case.

\end{proof}

\section{\textbf{Outer approximation}}\label{sec:outer_approximation}

Let us recall some notions from \cite{BCLV_subgroup_tests}. Let $\cF=\cF(S)$ be the free group with finite generating set $S$, and let  $B\subseteq\cF$ be a finite subset. For $Y,N\subseteq B$, let $C_B(Y,N)=\{A\subseteq B \mid Y\subseteq A\land N\cap A=\emptyset \}\subseteq \{0,1\}^B$. A subset $A\subseteq B$ is a \emph{pseudo subgroup} if there is a subgroup $K\leq \cF$ such that $B\cap K=A$, and we denote the set of them by $P\Sub(B)$. A random subset of $B$, say $\pi$, is a \emph{pseudo IRS} if it is supported on pseudo subgroups and for every two subsets $Y,N\subseteq B$ which satisfy  $sYs^{-1},sNs^{-1}\subseteq B$ for every $s\in S\sqcup S^{-1}$, we have 
\[
\forall s\in S\sqcup S^{-1}\ \colon \ \ \pi(C_B(Y,N))=\pi(C_B(sYs^{-1},sNs^{-1})).
\]
Denote by $P\IRS(B)\subseteq \Prob(\{0,1\}^B)$ the pseudo IRSs of $B$. 
Recall from \cite{BCLV_subgroup_tests} the notation $\pi\cap B$ for the pushforward of $\pi$ along the restriction map $R_B(K)=K\cap B$, and let $\widetilde{P\IRS}(B)$ consist of all  random subgroups $\pi$ such that $\pi\cap B\in P\IRS(B)$.
In Section 2.2 of \cite{BCLV_subgroup_tests}, it was shown that $\displaystyle{\bigcap_{B\subseteq \cF, |B|<\infty}  }\widetilde{P\IRS}(B)=\IRS(\cF)$.
\begin{rem}
    Recalling a footnote from the Introduction, we do not think of a random subset (or a random subgroup) as a random variable with values in subsets, but as the distribution itself --- this is a slight  abuse of terminology, but is consistent throughout this paper (and \cite{BCLV_subgroup_tests}), so we keep with it here.
\end{rem}

Here we define a natural subset of $P\IRS(B)$ which captures a local form of \eqref{eq:cont_condition}, which  was shown to be equivalent to RBS (Theorem \ref{thm:equiv_cond_to_RBS_via_cont_obs}). 
Let $W,W'\subseteq B$ be  subsets;  we think of these sets as ``windows''. A $(W,W')$-continuous observable on $B$ is a map $\psi\colon \{0,1\}^B\times \Sigma^B\to \Omega$ such that $\psi$ ``only depends on the $W,W'$ local view''; namely, if $(K_1,c_1),(K_2,c_2)\in \Sub(B,\Sigma)$ and $(K_1\cap W, c_1|_{W'})=(K_2\cap W, c_2|_{W'})$, then $\psi(K_1,c_1)=\psi(K_2,c_2)$. If $\gamma\in \cF$, let $\gamma.c\colon B\to \Sigma$ be 
\[
\forall x\in B\ \colon \ \ \gamma.c(x)=\begin{cases}
    c(\gamma^{-1}x) & \gamma^{-1}x\in B,\\
    \sigma & \gamma^{-1}x\notin B,
\end{cases}
\]
where $\sigma$ is some fixed element in $\Sigma$. If $\gamma\in \cF$ 
is such that $\gamma W\gamma^{-1},\gamma W'\subseteq B$, we can define a $(\gamma W\gamma^{-1},\gamma W')$-continuous observable $\psi^\gamma \colon \{0,1\}^B\times \Sigma^B\to \Omega$ by letting $\psi^\gamma(K,c)=\psi(\gamma^{-1} K\gamma\cap B,\gamma^{-1}.c)$; 
indeed, if $K_1\cap \gamma W\gamma^{-1}=K_2\cap \gamma W\gamma^{-1}$ and $c_1|_{\gamma W'}=c_2|_{\gamma W'}$, then $\gamma^{-1} K_1\gamma\cap  W =\gamma^{-1}K_2\gamma \cap  W$ and $\gamma^{-1}.c_1(w)=c_1(\gamma w)=c_2(\gamma w)=\gamma^{-1}.c_2(w)$ for every $w\in W'$, namely $\gamma^{-1}.c_1|_{W'}=\gamma^{-1}.c_2|_{W'}$, and thus
\[
\begin{split}
\psi^{\gamma}(K_1,c_1)&=\psi(\gamma ^{-1}K_1\gamma\cap B, \gamma^{-1}.c_1)\\
&=_{\psi\ {\rm is\ }(W,W')\ {\rm continuous}}\psi(\gamma ^{-1}K_2\gamma\cap B, \gamma^{-1}.c_2)\\
&=\psi^{\gamma}(K_2,c_2).
\end{split}
\]
If $A\subseteq \cF$ is such that $\bigcup_{\gamma\in A}\gamma W\gamma^{-1},\bigcup_{\gamma\in A} \gamma W'\subseteq B$, then we can define $\psi^A\colon \{0,1\}^B\times \Sigma^B\to \Omega^A$ by $\psi^A=\bigvee_{\gamma\in A} \psi^\gamma$, which is $(\bigcup_{\gamma\in A}\gamma W\gamma^{-1},\bigcup_{\gamma\in A} \gamma W')$-continuous. 

Let $K\subseteq B$ be a pseudo subgroup of $B$, and let $\Sigma$ be a finite set. A coloring $c\colon B\to \Sigma$ is $K$-invariant if for every  $x,y\in B$ for which $xy^{-1}\in K$ we have $c(x)=c(y)$. Let $P\Sub(B,\Sigma)\subseteq \{0,1\}^B\times \Sigma^B$ consist of all pairs $(K,c)$ where $K$ is a pseudo subgroup of $B$ and $c\colon B\to \Sigma$ is $K$-invariant. For a fixed $K\in P\Sub(B)$, let $\nu_K$ be the uniform measure over all $K$-invariant colorings $c\colon B\to \Sigma$ (this is the local analogue of $u_K$). 
If we abuse notation and denote by $_K\backslash^B$ the smallest equivalence classes  on $B$ induced by the relations $x\sim_K y$ if $xy^{-1}\in K$, then a $K$-invariant coloring is a map $c\colon _K\backslash^B\to \Sigma$. Given $\pi$ a pseudo IRS on $B$, let $\nu_\pi$ be the probability distribution over $P\Sub(B,\Sigma)$ which first samples $K\sim \pi$ and then a $K$-invariant coloring $c$ uniformly at random (this is the local analogue of $u_\pi$ for an IRS $\pi$); namely, $\nu_\pi=\Ex_{K\sim \pi}[{\bf 1}_K\times \nu_K]$. As usual, $P\colon P\Sub(B,\Sigma)\to P\Sub(B)$ is the projection to the first coordinate. 

Now, because there is no natural group action in this local setup, we cannot discuss generating observables. But, the whole point of Theorem \ref{thm:equiv_cond_to_RBS_via_cont_obs} was to translate the RBS condition to  the continuous non-generating condition \eqref{eq:cont_condition}. The two components of \eqref{eq:cont_condition} are $H_{u_\pi}(\psi|P)$ and $H_{u_\pi}(\phi_\Sigma|\psi^\cF\vee P)=\inf_{A\subseteq \cF, |A|<\infty}(H_{u_\pi}(\phi_\Sigma|\psi^A\vee P))$, and they have a local view version. If $\Id_\cF\in B$, let $\phi_\Sigma\colon \{0,1\}^B\times \Sigma^B\to \Sigma$ be $\phi_\Sigma(K,c)=c({\Id}_\cF)$.
\begin{defn}
Let $\Id_\cF\in B$ be a finite subset of $\cF$. A pseudo IRS $\pi$ of $B$ is  \emph{locally satisfying the RBS condition} if  for every three subsets $W,W',A\subseteq B$, finite sets $\Sigma,\Omega$  and function $\psi\colon \{0,1\}^B\times \Sigma^B\to \Omega$ satisfying 
\begin{enumerate}
    \item for all $\gamma\in A$ we have $\gamma W\gamma^{-1},\gamma W' \subseteq B$;
    \item $W'W'^{-1}\subseteq W$; 
    \item $|\Sigma|,|\Omega|\leq |B|$;
    \item $\psi$ is  $(W,W')$-continuous; 
\end{enumerate}
we have 
\begin{equation}\label{eq:local_cont_condition}
H_{\nu_\pi}(\psi|P)+H_{\nu_\pi}(\phi_\Sigma|\psi^A\vee P)\geq \log|\Sigma|.
\end{equation}
We denote the set of these by $P\RBS(B)$, and the set of random subgroups $\pi$ such that $\pi\cap B$ is in $P\RBS(B)$ by $\widetilde{P\RBS}(B)$.
\end{defn}

\begin{thm}\label{thm:outer_approximation}
    For every $\Id_\cF\in B\subseteq \cF$ where $|B|\leq \cF$, the set ${P\RBS}(B)$ of pseudo IRSs which locally satisfy RBS  is a compact convex polytope in $\RR^{\{0,1\}^B}$ whose defining half spaces are computable from $B$, and 
    \[
    \bigcap_{\substack{\Id_\cF\in B\subseteq \cF\\ |B|<\infty}}\widetilde{P\RBS}(B)=\IRS_\RBS(\cF),
    \]
    where $\IRS_\RBS(\cF)$ is the set of IRSs satisfying the RBS condition (Definition \ref{defn:RBS_for_IRSs}). In particular, $\IRS_\RBS(\cF)$ is weak* closed. 
\end{thm}
Before proving Theorem \ref{thm:outer_approximation} let us spell an immediate corollary of it:
\begin{cor}\label{cor:cosofic_IRSs_satisfy_RBS}
    Every cosofic IRS (Definition \ref{defn:soficity_and_cosoficity}) satisfies the RBS criterion (Definition \ref{defn:RBS_for_IRSs}).
\end{cor}
\begin{proof}
    By Definition \ref{defn:RBS_for_IRSs}, every finitely desecribed IRS of $\cF$ is in $\IRS_\RBS(\cF)$. By Theorem \ref{thm:outer_approximation}, $\IRS_\RBS(\cF)$ is weak* closed. Therefore, the weak* closure of $\IRS_\fd(\cF)$, which is the set of cosofic IRSs, is contained in $\IRS_\RBS(\cF)$. 
\end{proof}

\begin{proof}[Proof of Theorem \ref{thm:outer_approximation}]
    In \cite{BCLV_subgroup_tests}, we showed that $P\IRS(B)$ is a compact convex polytope with halfspaces computable from $B$. The set $P\RBS(B)$ is a subset of $P\IRS(B)$, and is defined by satisfying extra conditions of the form \eqref{eq:local_cont_condition}; these are linear inequalities on  $\pi\colon \{0,1\}^B\to \RR$, since:\footnote{The following is actually a more general phenomenon: The map from $\pi$ to $\nu_\pi$ is convex-linear in the sense that if, for $i=1,2$, $\pi^{(i)}$ are given and $t \in [0,1]$, then
$$\nu_{t \pi_1 + (1-t)\pi_2} = t\nu_{\pi_1} + (1-t)\nu_{\pi_2}.$$
Furthermore, Shannon entropy is convex linear as a function of the measure. That is $$H_{t\mu_1 + (1-t)\mu_2}(\alpha|\mathscr{F}) = tH_{\mu_1 }(\alpha|\mathscr{F}) +(1-t)H_{\mu_2 }(\alpha|\mathscr{F}) $$
in general (for any measures $\mu_1,\mu_2$ on the same space and any fixed $\alpha,\mathscr{F}$). Since the composition of convex linear functions is convex linear, it follows that the conditions of \eqref{eq:local_cont_condition} are indeed linear inequalities. }
    \begin{equation}\label{eq:interpretation_of_entropy_given_P}
    \begin{split}
H_{\nu_\pi}(\psi|P)&=\Ex_{K\sim \pi}[H_{{\bf 1}_K\times \nu_K}(\psi)]=\sum_{K\subseteq B} H_{{\bf 1}_K\times \nu_K}(\psi)\cdot \pi(K),\\
H_{\nu_\pi}(\phi_\Sigma|\psi^A\vee P)&=\Ex_{K\sim \pi}[H_{{\bf 1}_K\times \nu_K}(\phi_\Sigma|\psi^A)]=\sum_{K\subseteq B} H_{{\bf 1}_K\times \nu_K}(\phi_\Sigma|\psi^A)\cdot \pi(K).
    \end{split}
    \end{equation}
    Moreover, recalling the notation $\mu^\psi_y$ for conditional distributions from Section \ref{sec:basic_notions}, for $K\subseteq B$ we have
    \[
    \begin{split}
    \forall a\in \Omega \ \colon \ \ \psi_*{\bf 1}_K\times \nu_K(a)&=\Pro_{c\sim \nu_K}[\psi(K,c)=a]\\
    &=\frac{|\{c\colon _K\backslash^B\to \Sigma\mid \psi(K,c)=a\}|}{|\Sigma|^{ |_K\backslash^B|}}\ ,\\
   \forall a\in \Omega,\ \alpha\in \Omega^A \ \colon \ \ \phi_{\Sigma*}({\bf 1}_K\times \nu_K)^{\psi^A}_\alpha(a)&=  \Pro_{c\sim \nu_K}[\phi_\Sigma(K,c)=a|\psi^A(c)=\alpha]\\
   &=\frac{|\{c\colon  _K\backslash^B\to \Sigma\mid \phi_\Sigma(c)=a\land \psi^A(c)=\alpha\}|}{|\{c\colon  _K\backslash^B\to \Sigma\mid \psi^A(c)=\alpha\}|}\ .
    \end{split}
    \]
    These distributions are thus computable given $\psi,A$ and $K$, and it is straightforward to calculate their appropriate entropies. Thus, the coefficients  in \eqref{eq:interpretation_of_entropy_given_P} are computable from $\psi,A$ and $K$,
    which demonstrates that the inequality \eqref{eq:local_cont_condition} is  computable from them.
   Since there are only finitely many  $W,W',A,\Sigma,\Omega$ and $\psi$ for every $B$,\footnote{We are a bit sloppy here. There are infinitely many sets $\Omega,\Sigma$ such that $|\Omega|,|\Sigma|\leq |B|$. But, we can assume the sets we use are always of the form $\{1,...,k\}$ for $k\leq |B|$, and they will induce the same sets $P\RBS(B)$, as the restrictions depend only on the size of $\Sigma$ and $\Omega$ and not the sets themselves.} the sets $P\RBS(B)$ are defined by some finite set of extra computable linear inequalities, and are therefore compact convex computable polytopes.

    Assume the IRS $\pi$ satisfies the RBS criterion. Then, $\pi\cap B$ is a pseudo IRS by \cite{BCLV_subgroup_tests}. Let $W,W',A,\Sigma,\Omega$ and $\psi$  be as in the theorem assumptions. Let $\widetilde \psi \colon \Sub(\cF,\Sigma)\to \Omega$ be the following observable:
    \[
    \widetilde\psi(K,c)=\psi(K\cap B,c|_{B}).
    \]
    The observable $\widetilde \psi$ is continuous by definition, as it depends only on the restrictions of $K$ and $c$ to $B$. Since $W'W'^{-1}\subseteq W$, the distributions $\nu_{\pi\cap B}\cap (W,W')$ and $u_\pi \cap (W,W')$ are identical --- for $K\leq \cF$, the local cosets $_{\langle K\cap B\rangle}\backslash^B$ are a priori finer than the intersection $(_K\backslash^B)$, since there may be $x,y\in B$ such that $x\not\sim_{\langle K\cap B\rangle} y$, as there is no sequence of $z_1,...,z_t\in B$ for which $xz_1^{-1},z_1z_2^{-1},...,z_ty^{-1}\in K\cap B$, but $xy^{-1}\in K$. So, a  $K\cap B$-invariant coloring is a looser condition than being the restriction to $B$ of a $K$-invariant coloring. Nevertheless, this is not the case for the restrictions of these colorings to $W'$, as we forced $W\subseteq B$ to include all words of the form $xy^{-1}$ from $W'$. In particular, the distributions $u_\pi$ and $\nu_{\pi\cap B}$ agree when restricted to the $(W,W')$ local view. Furthermore, as $\gamma W\gamma ^{-1},\gamma W'\subseteq B$ for every $\gamma \in A$, we have that $\widetilde \psi^\gamma (K,c)=\psi^\gamma(K\cap B,c|_B)$, and thus $\widetilde \psi^A(K,c)=\psi^A(K\cap B,c|_B)$. All in all, we get that $H_{\nu_{\pi\cap B}}(\psi|P)=H_{u_\pi}(\widetilde\psi|P)$ and $H_{\nu_{\pi\cap B}}(\phi_\Sigma|\psi^A\vee P)=H_{u_\pi}(\phi_\Sigma|\widetilde\psi^A\vee P)$. Since $\pi$ satisfies RBS, by Theorem \ref{thm:equiv_cond_to_RBS_via_cont_obs} we have 
    \[
    \begin{split}
\log|\Sigma|&\leq H_{u_\pi}(\widetilde\psi|P)+H_{u_\pi}(\phi_\Sigma|\widetilde\psi^\cF\vee P)\\
&\leq H_{u_\pi}(\widetilde\psi|P)+H_{u_\pi}(\phi_\Sigma|\widetilde\psi^A\vee P)\\
&\leq H_{\nu_{\pi\cap B}}(\psi|P)+H_{\nu_{\pi\cap B}}(\phi_\Sigma|\psi^A\vee P).
    \end{split}
    \]
As this analysis was independent of the specific $W,W',A,\Sigma,\Omega$ and $\psi$ we started with,  $\pi\cap B$ satisfies all the defining relations of $P\RBS(B)$ for every $B$, and is thus in $  \displaystyle{\bigcap_{\substack{\Id_\cF\in B\subseteq \cF\\ |B|<\infty}}\widetilde{P\RBS}(B)}$. 

Assume the IRS $\pi$ does not satisfy RBS. Then, by Theorem \ref{thm:equiv_cond_to_RBS_via_cont_obs}, there is some set $\Sigma$ and a continuous observable $\psi\colon \Sub(\cF,\Sigma)\to \Omega$ such that 
\[
H_{u_\pi}(\psi|P)+H_{u_\pi}(\phi_\Sigma|\psi^\cF\vee P)<\log|\Sigma|.
\]
By \cref{clause:limit_pregiven_info} in Fact \ref{fact:basic_prop_entropy}, $H_{u_\pi}(\phi_\Sigma|\psi^\cF\vee P)=\inf_{A\subseteq \cF, |A|<\infty}H_{u_\pi}(\phi_\Sigma|\psi^A\vee P)$, and there is thus some large enough (yet finite) $A$ for which 
\[
H_{u_\pi}(\psi|P)+H_{u_\pi}(\phi_\Sigma|\psi^A\vee P)<\log|\Sigma|.
\]
Since $\psi$ is continuous, there are finite sets $W,W'\subseteq \cF$ such that $\psi(K_1,c_1)=\psi(K_2,c_2)$ whenever $(K_1\cap W,c_1|_{W'})=(K_2\cap W,c_2|_{W'})$. Without loss of generality, we can assume $W$ contains $W'W'^{-1}$.  Let $B$ be a large enough finite subset of $\cF$ such that $W,W',\gamma W\gamma^{-1},\gamma W'\subseteq B$ for every $\gamma\in A$, and $|\Sigma|,|\Omega|\leq |B|$. Let $\widehat \psi\colon \{0,1\}^B\times \Sigma^B\to \Omega$ be 
\[
\widehat \psi(K,c)=\begin{cases}
    \psi(\widetilde K, \widetilde c) & \exists (\widetilde K, \widetilde c)\in \Sub(\cF,\Sigma) \colon K\cap W=\widetilde K\cap W\land c|_{W'}=\widetilde c|_{W'},\\
    \omega & {\rm otherwise},
\end{cases}
\]
where $\omega$ is some fixed element of $\Omega$. In the notations of the first direction of this proof, we chose $\widehat \psi$ so that $\widetilde {(\widehat \psi)}=\psi$. As so, by our choice of $B$ and the analysis in the proof of the first direction, we have $H_{\nu_{\pi\cap B}}(\widehat\psi|P)=H_{u_\pi}(\psi|P)$ and $H_{\nu_{\pi\cap B}}(\phi_\Sigma|\widehat\psi^A\vee P)=H_{u_\pi}(\phi_\Sigma|\psi^A\vee P)$ and thus 
\[
H_{\nu_{\pi\cap B}}(\widehat\psi|P)+H_{\nu_{\pi\cap B}}(\phi_\Sigma|\widehat\psi^A\vee P)<\log|\Sigma|.
\]
Hence, $\pi\cap B\notin P\RBS(B)$ and we are done. 
\end{proof}

To prove Theorem \ref{thm:main} in the next section, we would need to run linear programs on the polytopes $P\RBS(B)$. As opposed to \cite{BCLV_subgroup_tests}, where all equations (and inequalities)  defining the pseudo IRSs $P\IRS(B)$ were integral, the sets  $P\RBS(B)$ have irrational coefficients in their defining inequalities. 
This means that the algorithm calculating the optimum over these sets would first need to  approximate them, and that leads to errors. 
There are standard continuity results for linear programming (cf.\ \cite{wets2009continuity}), which can still guarantee the algorithmic outer approximation we need. But, we decided to be more direct and define an alternative outer approximation to $P\RBS(B)$ which uses only rational coefficients in its defining inequalities. 

Given a real number $a$ and a positive integer $n$, let $\rho(a,n)=\frac{\lceil 10^n \cdot a\rceil}{10^n}$ --- this is the $n$-digits upper decimal expansion of $a$. Similarly, for an inequality $\alpha \colon \sum a_iX_i+b\geq 0$ with real coefficients $a_i,b$ and variables $X_i$, we let $\rho(\alpha,n)\colon \sum\rho(a_i,n)X_i+\rho(b,n)\geq 0$. Note that  a non-negative assignment $\forall i, X_i=v_i\geq 0$ which satisfies $\alpha$ must also satisfy $\rho(\alpha,n)$ for every $n$, since $\rho(b,n)\geq b$ and $\rho(a_i,n)X_i\geq a_i X_i$. Let $P\RBS(B,n)$ be $P\IRS(B)$ with the inequalities $\rho(\eqref{eq:local_cont_condition},n)$\footnote{As the constant $\log|\Sigma|$ is on the right hand side of \eqref{eq:local_cont_condition}, one first needs to move it to the left and then apply $\rho(\cdot,n)$.} instead of \eqref{eq:local_cont_condition} for every relevant tuple $W,W',A,\Sigma,\Omega,\psi$. For $B\subseteq C\subseteq \cF$, and $A\subseteq C$, recall the notation $R_{B\subseteq C}(A)=A\cap B\subseteq B$ for the restriction map from $C$ to $B$. The following is quite immediate, and we leave it without a proof:
\begin{fact}\label{fact:rational_outer_approximation}
    Let $\Id_\cF\in B\subseteq C\subseteq\cF$ and $n\in \NN$ such that $|B|,|C|<\infty$. Then:
    \begin{itemize}
    \item $P\RBS(B,n+1)\subseteq P\RBS(B,n)$;
      \item $R_{B\subseteq C*}P\RBS(C)=\{R_{B\subseteq C*}\pi\mid\pi\in P\RBS(C)\}\subseteq P\RBS(B)$;
        \item $R_{B\subseteq C*}P\RBS(C,n)=\{R_{B\subseteq C*}\pi\mid\pi\in P\RBS(C,n)\}\subseteq P\RBS(B,n)$;
        \item $\bigcap_{n=1}^\infty P\RBS(B,n)=P\RBS(B)$.
    \end{itemize}
    By combining these observations  with Theorem \ref{thm:outer_approximation}, and by denoting $\widetilde{P\RBS}(B,n)$ for all random subgroups $\pi\in \Prob(\Sub(\cF))$ satisfying $R_{B*}\pi=\pi\cap B\in P\RBS(B,n)$, we have
    \[
\bigcap_{\substack{\Id_\cF\in B\subseteq \cF\\ |B|<\infty}}\widetilde{P\RBS}(B,|B|)=\IRS_\RBS(\cF).
    \]
\end{fact}

\section{\textbf{Proof of Theorem \ref{thm:main}}}\label{sec:proof_of_main}
Recall the notion of a subgroup test $\cT$ (Section \ref{sec:subgroup_tests}) consisting of challenges $\{D_i\}_{i=1}^m$ and a rational probability distribution $\mu\in \Prob [m]$, and specifically its value versus an IRS defined in  equation \eqref{eq:value_of_a_subgroup_test}. As challenges are continuous functions $D_i\colon \Sub(\cF)\to \{0,1\}$, there are finite subsets  $\Delta_i\subseteq \cF$  and functions $d_i\colon P\Sub(\Delta_i)\to \{0,1\}$ such that $D_i(K)=d_i(K\cap \Delta_i)$. Let $\Delta=\bigcup_{i=1}^m\Delta_i$ be the \emph{mutual continuity domain} of $\cT$. For every finite subset $B\subseteq \cF$ such that $\Delta\subseteq B$, one can define a value of a  random pseudo subgroup $\pi$ of $B$ by letting 
\begin{equation}\label{eq:value_of_a_pseudo_random_subgp}
    \val(\cT,\pi)=\Ex_{i\sim \mu}\Ex_{K\sim \pi}[d_i(K\cap \Delta_i)].
\end{equation}
It is straightforward to check that $\val(\cT,\pi)$ as in \eqref{eq:value_of_a_subgroup_test} agrees with $\val(\cT,\pi\cap B)$ as in \eqref{eq:value_of_a_pseudo_random_subgp}, and so it makes sense to use the same notation. Let $\val_\RBS(\cT)$ be the supremum of $\val(\cT,\pi)$ when running over all IRSs  $\pi$ which satisfy the RBS criterion. 

\begin{thm}\label{thm:algorithmic_outer_approximation}
    There is an algorithm which takes as input a subgroup test $\cT$ and outputs a non-increasing sequence of rationals $\{\rho_r\}_{r=r_0}^\infty$ such that $\displaystyle{\lim_{r\to \infty} }\rho_r=\val_\RBS(\cT)$.
\end{thm}

\begin{proof}
    Let $\Delta$ be the mutual continuity domain of $\cT$. Now, for every finite set $B\subseteq \cF$ which contains $\Delta$ and $\Id_\cF$, $\val(\cT,\cdot)$ induces a linear functional  from $\RR^{\{0,1\}^B}$ to $\RR$. Maximizing this functional over a compact computable rational polytope of $\RR^{\{0,1\}^B}$ can be done in finite time (this is exactly the task of linear programing). Hence, finding $\rho_B=\max\{\val(\cT,\pi)\mid \pi\in P\RBS(B,|B|)\}$ can be calculated in  finite time.
     As $\Delta$ is a finite set of words in the free group, there is an integer $r_0$ which is the length of the longest word in $\Delta$. So, $\Delta\subseteq B_r(\Id_\cF)$ for every $r\geq r_0$. Let $\rho_r=\rho_{B_r}$ for every $r\geq r_0$. By Fact \ref{fact:rational_outer_approximation} and the fact \eqref{eq:value_of_a_subgroup_test} and \eqref{eq:value_of_a_pseudo_random_subgp} agree,  the limit of the sequence $\rho_r$ is indeed $\val_\RBS(\cT)$.
\end{proof}

By combining Theorem \ref{thm:algorithmic_outer_approximation} with Fact \ref{fact:inner} and Fact \ref{fact:sof_val_is_RE_hard}, we deduce Theorem \ref{thm:main}. \qed

\bibliographystyle{plain}
\bibliography{Bibliography}

\end{document}